\documentclass{article}%
\usepackage{amsfonts}
\usepackage{mathrsfs}
\usepackage{graphicx}
\usepackage{amsmath}
\usepackage{amssymb}%
\usepackage{xcolor}

\usepackage[
bookmarks=true,         %generates bookmarks for all entries in the "Table of Contents"
bookmarksnumbered=true, %includes chapter-/header-/section-/subsection-/... -
colorlinks=true, pdfstartview=FitV, linkcolor=blue, citecolor=blue,
urlcolor=blue]{hyperref}

\providecommand{\U}[1]{\protect\rule{.1in}{.1in}}
\newtheorem {theorem}{Theorem}[section]
\newtheorem {proposition}{Proposition}[section]

\newtheorem{lemma}{Lemma}[section]

\newtheorem{remark}{Remark}[section]

\newenvironment{proof}[1][Proof]{\textbf{#1.} }{\
\rule{0.5em}{0.5em}}

\newcommand{\eq}[1]{$(\ref{#1})$}
\newcommand{\ignore}[1]{}{}
\newcommand{\lbl}{\label}

\begin{document}

\begin{center}
{\Large Self-Normalized Moderate Deviations for Degenerate {\it U}-Statistics }%and the LIL for Degenerate {\it U}-Statistics of Order $2$ }

\bigskip

\centerline{\today}

\bigskip Lin Ge$^{a}$, Hailin Sang$^{b}$\footnote{This research is partially supported by the Simons Foundation Grant 586789, USA.}, Qi-Man Shao$^{c}$\footnote{This research is partially supported by the National Nature Science Foundation of China NSFC 12031005 and
Shenzhen Outstanding Talents Training Fund, China.}

\bigskip$^{a}$ Division of Arts and Sciences, Mississippi State University at Meridian,
Meridian, MS 39307, USA. E-mail address: lge@meridian.msstate.edu

\bigskip$^{b}$ Department of Mathematics, University of Mississippi,
University, MS 38677, USA. E-mail address: sang@olemiss.edu

\bigskip$^{c}$ Department of Statistics and Data Science, Shenzhen International Center for Mathematics,
Southern University of Science and Technology, Shenzhen, Guangdong 518055, P.R. China. E-mail address:
shaoqm@sustech.edu.cn
\end{center}
\begin{abstract}
In this paper, we study self-normalized moderate deviations for degenerate { $U$}-statistics of order $2$. Let  $\{X_i, i \geq 1\}$ be i.i.d. random variables and consider symmetric and degenerate kernel functions in the form $h(x,y)=\sum_{l=1}^{\infty} \lambda_l g_l (x) g_l(y)$,  where $\lambda_l > 0$,  $E g_l(X_1)=0$, and $g_l (X_1)$ is in the domain of attraction of a normal law for all $l \geq 1$. Under the condition $\sum_{l=1}^{\infty}\lambda_l<\infty$ and some truncated conditions for $\{g_l(X_1): l \geq 1\}$, we show that $ \text{log} P({\frac{\sum_{1 \leq i \neq j \leq n}h(X_{i}, X_{j})} {\max_{1\le l<\infty}\lambda_l V^2_{n,l} }} \geq x_n^2) \sim - { \frac {x_n^2}{ 2}}$ 
for $x_n \to \infty$ and $x_n =o(\sqrt{n})$, where $V^2_{n,l}=\sum_{i=1}^n g_l^2(X_i)$.  As application, a law of the iterated logarithm is also obtained.

\end{abstract}
\noindent Keywords:  moderate deviation, degenerate {\it U}-statistics,  law of the iterated logarithm, self-normalization %degenerate {\it U}-statistics

\noindent  {\textit{MSC 2010 subject classification}: 60F15, 60F10, 62E20}

\numberwithin{equation}{section}
\section{Introduction and main results}
Last three decades has witnessed significant developments on self-normalized limit theory,
especially on large deviations, Cram\'er type moderate deviations, and the law of the iterated logarithm.
  Comparing with the classical limit theorems, these self-normalized limit theorems usually require much less moment assumptions.

Let $X, X_1, X_2, \dots$ be independent identically distributed (i.i.d.) random variables. Set
\begin{eqnarray*}
S_n=\sum_{i=1}^nX_i \ \  \textrm{and} \ \ V_n^2=\sum_{i=1}^n X_i^2.
\end{eqnarray*}

Griffin and Kuelbs \cite{GK} obtained a law of the iterated logarithm (LIL) for the self-normalized
sum of i.i.d. random variables with distributions in the domain of attraction of a
normal or stable law. They proved that
\begin{enumerate}
\item If $EX=0$ and $X$ is in the domain of attraction of a normal law, then
\begin{equation}
\limsup_{n\rightarrow \infty}\frac{S_n}{V_n (2\log\log n)^{1/2}}=1\;\;a.s. \lbl{0.1}
\end{equation}
\item If $X$ is symmetric and is in the domain of attraction of a stable law, then there exists a positive constant $C$ such that
\begin{equation}
\limsup_{n\rightarrow \infty}\frac{S_n}{V_n (2\log\log n)^{1/2}}=C\;\;a.s. \lbl{0.2}
\end{equation}
\end{enumerate}

Shao \cite{Shao97} obtained the following self-normalized moderate deviations and specified the constant $C$ in \eq{0.2}. Let $\{x_n, n\ge 1\}$ be a sequence of positive numbers such that $x_n\rightarrow \infty$ and $x_n=o(\sqrt{n})$ as $n\rightarrow \infty$.
\begin{enumerate}
\item If $EX=0$ and $X$ is in the domain of attraction of a normal law, then
\begin{equation}
\lim_{n \rightarrow \infty} x_n^{-2} \log P \left (  \frac{S_n}{V_n} \geq x_n\right ) = - \frac{1}{2}. \lbl{0.3}
\end{equation}
\item If $X$ is in the domain of attraction of a stable law such that $EX=0$ with index $1 < \alpha <2$, or $X_1$ is symmetric with index $\alpha=1$, then
\begin{equation*}
\lim_{n \rightarrow \infty} x_n^{-2} \log P \left (  \frac{S_n}{V_n} \geq x_n\right ) = - \beta(\alpha, c_1, c_2),
\lbl{0.4}
\end{equation*}
\end{enumerate}
where $\beta(\alpha, c_1, c_2)$ is a constant depending  on the tail distribution, see \cite{Shao97} for an explicit
definition.

Shao \cite{Shao99} refined \eq{0.3} and obtained the following Cram\'er type moderate deviation theorem under a finite third moment: If $EX=0$ and $E|X|^3 < \infty$, then
\begin{equation}\label{0.5}
 { P(S_n/V_n \geq x_n) \over P( Z\geq x_n)}
 \to 1
 \end{equation}
for any $x_n \in [0, o (n^{1/6}))$, where $Z$ is the standard normal random variable.

Jing, Shao and Wang \cite{JingShaoWang} further extended \eq{0.5} to general independent random variables under a Lindeberg type condition, while Shao and Zhou \cite{ShaoZhou} established the result for self-normalized non-linear statistics which include {\it U}-statistics as a special case.

The {\it U}-statistics were introduced by Halmos \cite {Halmos} and Hoeffding \cite {Hoeffding}. The LIL for nondegenerate {\it U}-statistics was obtained by Serfling \cite{S}. The LIL for degenerate {\it U}-statistics was studied by Dehling, Denker and Philipp (\cite{DDP1}, \cite{DDP2}), Dehling \cite{D},  Arcones and Gin\a'{e} \cite{AG}, Teicher \cite{T}, Gin\a'{e} and Zhang \cite{GZ}, and others.  Gin\a'{e}, Kwapie\a'{n}, Lata\l a and Zinn \cite{GKLZ} gave necessary and sufficient conditions for the LIL of degenerate {\it U}-statistics of order 2, which was extended to any order by Adamczak and  Lata\l a \cite{AL}.

The main purpose of this paper is to  study  the self-normalized moderate deviations and the LIL for
degenerate {\it U}-statistics of order $2$. Let
%symmetric and degenerate {\it U}-statistics of order 2 defined as
\begin{eqnarray*}
U_n = \frac{1}{n(n-1)} \sum_{1 \leq i \neq j \leq n} h(X_i, X_j) , \lbl{0.6}
\end{eqnarray*}
where
\begin{eqnarray} \label{1.2}
h(x,y)=\sum_{l=1}^{\infty} \lambda_l g_l (x) g_l(y).
\end{eqnarray}

A motivation example for the LIL is the one with the kernel $h(x,y)=xy$. Obviously, $V_n^2/(2V_n^2\log\log n)\rightarrow 0$. Then by \eq{0.1} and \eq{0.2}
%based on the results of Griffin and Kuelbs \cite{GK} above, we have
\begin{enumerate}
	\item If $EX=0$ and $X$ is in the domain of attraction of a normal law, then
\begin{align}
&\limsup_{n\rightarrow \infty}\frac{1}{2V_n^2\log\log n}\bigg|\sum_{1 \leq i\ne j\le n}X_iX_j\bigg | \nonumber\\
&=\limsup_{n\rightarrow \infty}\left[\frac{S_n}{V_n (2\log\log n)^{1/2}}\right]^2=1\;\;a.s. \label {1.1aa}
\end{align}
\item If $X$ is symmetric and is in the domain of attraction of a stable law, then there exists a positive constant $C$ such that
\begin{align*}
&\limsup_{n\rightarrow \infty}\frac{1}{2V_n^2\log\log n}\left|\sum_{1 \leq i\ne j\le n}X_iX_j\right|\\
&=\limsup_{n\rightarrow \infty}\left[\frac{S_n}{V_n (2\log\log n)^{1/2}}\right]^2=C^2\;\;a.s.
\end{align*}
\end{enumerate}

\ignore{
In general, a symmetric and degenerate kernel can be written as
\begin{eqnarray*} \label {1.2}
h(x,y)=\sum_{l=1}^{\infty} \lambda_l g_l (x) g_l(y).
\end{eqnarray*}
Consider
}
For the general degenerate kernel $h$ defined in \eq{1.2}, we have

\begin{eqnarray*} \label {1.3}
n (n-1) U_n
 &=&  \sum_{l=1}^{\infty} \lambda_l \sum_{1 \leq i \neq j \leq n} g_l(X_i) g_l(X_j) \nonumber \\
&=&\sum_{l=1}^{\infty} \lambda_l \bigg (  \Big (\sum_{i=1}^n g_l (X_i)  \Big  )^2 - \sum_{i=1}^n g_l^2(X_i) \bigg ).
\end{eqnarray*}
Suppose that $g_l(X)$ is in the domain of attraction of a normal law for every  $l \geq 1$. Then $L_l(x) := E g_l^2(X_1)I (|g_l(X_1)| \leq x)$  is a slowly varying function for all $l \geq 1$ as $x \rightarrow \infty$. Let $\{x_n, n\ge 1\}$ be a sequence of positive numbers  such that $x_n\rightarrow \infty$ and $x_n=o(\sqrt{n})$ as $n \rightarrow \infty$. For each $l \geq 1$, set
\begin{eqnarray} \label {bznl}
b_l &=& \inf \Big \{ x \geq 1:  L_l (x)>0 \Big  \}, \nonumber \\
z_{n, l} &=& \inf \bigg \{ s : s \geq b_l +1, \frac{ L_l(s)}{s^2} \leq \frac{x_n^2}{n} \bigg \}.
\end{eqnarray}
\ignore{
Since $L_l(s)/s^2 \rightarrow 0$ as $s \rightarrow \infty$ and $ L_l(x)$ is right continuous, then $z_{n,l} \rightarrow \infty$ and for all $n$ sufficiently large,
\begin{eqnarray} \label {nznl}
n  L_l(z_{n, l})= x_n^2 z_{n, l}^2.
\end{eqnarray}
}
Write
\begin{eqnarray*}\label {AB1.5a}
W_n = \frac{n(n-1) U_n}{\max_{1 \leq l < \infty} \lambda_l \sum_{i=1}^n g_l^2(X_i)}.
\end{eqnarray*}

%-------------------------------------------------------------------------------- theorems

\noindent
We have the following self-normalized moderate deviation.
\begin{theorem}\label{moderate}
Let $Eg_l(X)=0$ and $\lambda_l \geq 0$ for every $l \geq 1$.  %(or replacing $\lambda_l$ with $|\lambda_l|$) for all $1 \leq l <\infty$. Assume that
\begin{eqnarray}\label {a1}
\sup_{ x \in R }\frac{\sum_{l=m+1}^{\infty}\lambda_l g^2_l(x)}{\sum_{1 \leq l <\infty}\lambda_l g^2_l(x)} \rightarrow 0 \ \  \textrm{as} \ \ m \rightarrow \infty
\end{eqnarray}
and
\begin{eqnarray}\label {a2}
\lim_{n \rightarrow \infty}\frac{E g_l(X)I(|g_l(X)| \leq z_{n,l})g_k(X)I(|g_k(X)| \leq z_{n,k})}{\sqrt{L_l(z_{n,l})L_k(z_{n,k})}} \rightarrow 0
\end{eqnarray}
for any $l \neq k$. Then for $x_n \rightarrow \infty$ and $x_n=o(\sqrt{n})$,
\begin{eqnarray*}\label {th1}
\lim_{n \rightarrow \infty} x_n^{-2} \log P \left ( W_n \geq x^2_n \right ) = - \frac{1}{2}.
\end{eqnarray*}
\end{theorem}

As an application, we have the following self-normalized LIL.
\begin{theorem} \label {LIL}
Under the assumptions in Theorem \ref{moderate},  and instead of \eqref{a1}, we assume that
for each $l \in [1, \infty)$, there is a constant $c_l >0$ such that
\begin{eqnarray}\label{a1'}
\sup_{x \in R}\frac{\lambda_l g_l^2(x)}{\sum_{l=1}^{\infty}\lambda_l g^2_l(x)} \leq c_l \ \ and \ \   \sum_{l=1}^{\infty}  c_l < \infty.
\end{eqnarray}
Then
\begin{eqnarray} \label {th2}
\limsup_{n \rightarrow \infty} \frac{W_n}{ \log \log n} = 2 \ \ a.s.
\end{eqnarray}
\end{theorem}

\begin{remark}
We use an example to show that (\ref{a2}) can't be removed. Let $g_1(x)=x$, $g_2(x)=x^3$ and $\lambda_1=\lambda_2=1$, $\lambda_l =0$ for $ l \geq 3$. Let $X$ be a Rademacher random variable. Then $n(n-1) U_n =\sum_{1 \leq i \neq j \leq n}(X_iX_j+X_i^3X_j^3)=2\sum_{1 \leq i \neq j \leq n}X_iX_j$ and $W_n=2\sum_{1 \leq i \neq j \leq n}X_iX_j/\sum_{i=1}^n X_i^2 $. By (\ref{1.1aa}),\\ $\limsup_{n \rightarrow \infty}W_n/\log \log n =4$ a.s. which contradicts (\ref{th2}).
\end{remark}

\section{Proofs}
\setcounter{equation}{0}

In the proofs of theorems, we will use the following properties for the slowly varying functions $g_l$ (e.g., Bingham et al. \cite{BGT}). As $x \rightarrow \infty$,
\begin{eqnarray}\label {Pgxa}
P (|g_l(X)| \geq x) =o(L_l(x)/x^2),
\end{eqnarray}
\begin{eqnarray}\label {Egx}
E |g_l(X)| I (|g_l(X)| \geq x)= o(L_l(x)/x),
\end{eqnarray}
\begin{eqnarray}\label {Egpx}
E |g_l(X)|^p I (|g_l(X)| \leq x)=o(x^{p-2}L_l(x)), \ \  p>2.
\end{eqnarray}

Since $L_l(s)/s^2 \rightarrow 0$ as $s \rightarrow \infty$ and $ L_l(x)$ is right continuous, hence in (\ref {bznl}), $z_{n,l} \rightarrow \infty$ and for all $n$ sufficiently large,
\begin{eqnarray} \label {nznl}
n  L_l(z_{n, l})= x_n^2 z_{n, l}^2.
\end{eqnarray}

%   -------------------------------------------------------------------------------------- properties

\subsection{The upper bound of Theorem \ref{moderate}}

For each $l \geq 1$ and $i \geq 1$, denote the truncated function
\begin{eqnarray*}\label {J1.5}
\bar{g}_{l}(X_i) = g_l(X_i) I \left (|g_l(X_i)|\leq z_{n , l} \right ).
\end{eqnarray*}
Since $E g_l(X_i)=0$, we have
\begin{align}\label{barest}
E \bar{g}_l(X_i)= o(L_l(z_{n,l})/z_{n,l})= o(x_n \sqrt{L_l(z_{n,l}})/\sqrt{n})
\end{align}
by (\ref{Egx}) and (\ref{nznl}).
For each $l \geq 1$, write
\begin{eqnarray*} \label {YnVn}
Y_{n,l} = \sum_{i=1}^n g_l(X_i),  \ \    \bar{Y}_{n,l}= \sum_{i=1}^n \bar{g}_l(X_i) \ \  \textrm{and} \ \  V^2_{n,l} = \sum_{i=1}^n g_l^2(X_i).
\end{eqnarray*}
By condition (\ref{a1}), for each $0<\varepsilon<1$, there exists $1 \leq m < \infty$ such that
\begin{eqnarray}\label{a1b}
m \max_{1 \leq l < \infty}\lambda_l V^2_{n,l}\geq \sum_{l=1}^m \lambda_l V^2_{n,l} \geq (1-\varepsilon)\sum_{l=1}^{\infty}\lambda_l V^2_{n,l}.
\end{eqnarray}
Hence for $x_n \rightarrow \infty$,
\begin{eqnarray} \label {AC2.1}
P \left ( W_n \geq (1+\varepsilon) x_n^2 \right ) \leq  P \bigg ( \frac{\sum_{l=1}^{\infty} \lambda_l Y_{n,l}^2}{\max_{1 \leq l <\infty} \lambda_l V^2_{n,l}} \geq x^2_n \bigg ) .
\end{eqnarray}
Observe that for any random variables $\{U_l\}_{l=1}^{\infty}$ and $\{ Z_l \}_{l=1}^{\infty}$ and any constants $x>0$ and $0<a < 1/3$, by the Cauchy inequality
\begin{eqnarray}\label{XY}
&& P\left  ( \sum_l (U_l+Z_l)^2  \geq x \right  ) \nonumber \\
%%+\sum_l Z_l^2 + 2 \left(\sum_l U_l^2 \right)^{1/2}  \left(\sum_l Z_l^2 \right)^{1/2}  \geq x \right  ) \nonumber \\
&\leq&  P\left (\sum_l U_l^2 \geq (1-a)^2 x \right ) +  P\left (\sum_l Z_l^2  \geq a^2 x \right ) \nonumber\\
%%&&+  P\Bigg ( 2 \left(\sum_l U_l^2 \right)^{1/2}  \left(\sum_l Z_l^2 \right)^{1/2} \geq 2a x \Bigg ) \nonumber\\
%%&\leq& 2  P\left (\sum_l U_l^2 \geq (1-3a) x \right ) +  P\left ( \sum_l Z_l^2  \geq a x \right ) \nonumber\\
%%&&+  P\Bigg (  \{(1-3a)x \}^{1/2}\left(\sum_l Z_l^2 \right)^{1/2}  \geq a x \Bigg ) \nonumber \\
&\leq&   P\left (\sum_l U_l^2 \geq (1-2a) x \right ) +  P\left (\sum_l Z_l^2  \geq a^2 x \right ).
\end{eqnarray}
%\red{note: the proof is simplified}

For any $0 < \varepsilon < 1/3$, by (\ref{AC2.1}) and (\ref{XY}),
\begin{eqnarray}\label {m2.7}
&& P\bigg (W_n \geq (1+\varepsilon)x_n^2  \bigg )\nonumber\\
&\leq&   P\bigg ( \frac{\sum_{l=1}^{\infty} \lambda_l \left (  \sum_{i=1}^{n} g_l(X_i) I \left (|g_l(X_i)| > z_{n,l} \right )  \right )^2}{\max_{1 \leq l <\infty}\lambda_l V^2_{n,l}} \geq  \varepsilon^2 x_n^2 \bigg) \nonumber\\
&&+ P \bigg ( \frac{ \sum_{l=1}^{\infty} \lambda_l \bar{Y}_{n,l}^2}{\max_{1 \leq l <\infty}\lambda_l V^2_{n,l}} \geq    (1- 2\varepsilon)  x_n^2  \bigg ) .
\end{eqnarray}
For any integer $m \geq 1$ and any constant $C_1>0$ with $C_1 \varepsilon<1$,
\begin{eqnarray}\label {2.8A}
&& P \bigg ( \frac{ \sum_{l=1}^{\infty} \lambda_l \bar{Y}_{n,l}^2}{\max_{1 \leq l <\infty}\lambda_l V^2_{n,l}} \geq    (1- 2\varepsilon)  x_n^2  \bigg ) \nonumber\\
&\leq&P \bigg  (\max_{1 \leq l <\infty} \lambda_l V^2_{n,l} \leq \frac{n}{\varepsilon}\sum_{l=m+1}^{\infty}\lambda_l L_l(z_{n,l})\bigg )\nonumber\\
&&+  P \bigg  (\max_{1 \leq l <\infty} \lambda_l V^2_{n,l} \leq (1-\varepsilon)n \max_{1 \leq l \leq m}\lambda_l L_l(z_{n,l})\bigg )\nonumber\\
&& +  P \bigg ( \frac{ \sum_{l=m+1}^{\infty} \lambda_l \bar{Y}_{n,l}^2}{(n/\varepsilon)\sum_{l=m+1}^{\infty}\lambda_l L_l(z_{n,l})} \geq C_1 \varepsilon  (1- 2\varepsilon)  x_n^2  \bigg )\nonumber\\
&& +    P \bigg ( \frac{ \sum_{l=1}^{m} \lambda_l \bar{Y}^2_{n,l}}{(1-\varepsilon)n\max_{1 \leq l \leq m}\lambda_l L_l(z_{n,l})} \geq  (1-C_1\varepsilon)  (1- 2\varepsilon)  x_n^2  \bigg ).
\end{eqnarray}
Applying (\ref{2.8A}) to (\ref{m2.7}), we have
\begin{eqnarray}\label {m2.10}
&&P\bigg ( W_n \geq (1+\varepsilon) x_n^2  \bigg ) \nonumber \\
&\leq& P \bigg  (\max_{1 \leq l <\infty} \lambda_l V^2_{n,l} \leq \frac{n}{\varepsilon}\sum_{l=m+1}^{\infty}\lambda_l L_l(z_{n,l})\bigg )\nonumber\\
&&+  P \bigg  (\max_{1 \leq l <\infty} \lambda_l V^2_{n,l} \leq (1-\varepsilon)n \max_{1 \leq l \leq m}\lambda_l L_l(z_{n,l})\bigg )\nonumber\\
&& + P\bigg ( \frac{\sum_{l=1}^{\infty} \lambda_l \left (  \sum_{i=1}^{n} g_l(X_i) I \left (|g_l(X_i)| > z_{n,l} \right )  \right )^2}{\max_{1 \leq l <\infty}\lambda_l V^2_{n,l}} \geq  \varepsilon^2 x_n^2 \bigg) \nonumber\\
&& + P \bigg ( \frac{ \sum_{l=m+1}^{\infty} \lambda_l \bar{Y}_{n,l}^2}{n\sum_{l=m+1}^{\infty}\lambda_l L_l(z_{n,l})} \geq C_1 (1- 2\varepsilon)  x_n^2  \bigg )\nonumber\\
&& +  P \bigg ( \frac{ \sum_{l=1}^{m} \lambda_l \bar{Y}^2_{n,l}}{n\max_{1 \leq l \leq m}\lambda_l L_l(z_{n,l})} \geq  (1-C_1\varepsilon)  (1- 2\varepsilon)^2  x_n^2 \bigg)  \nonumber\\
&:=& I_{1,1} +I_{1,2} +I_2 +I_3 +I_4.
\end{eqnarray}

% ------------------------------------------------------------------------------------------------------- subsection

\subsection{Estimation of $I_{1,1}$ and $I_{1,2}$}

% -------------------------------------------------------------------------------------------------------- proposition

\begin{proposition}\label{proposition1}
For $m \geq 1$ sufficiently large,
\begin{eqnarray}\label {I1(1)}
I_{1,1}=P \left (\max_{1 \leq l <\infty} \lambda_l V^2_{n,l }\leq \frac{n}{\varepsilon}\sum_{l=m+1}^{\infty}\lambda_l L_l(z_{n,l})\right )  \leq \exp (-2 x_n^2)
\end{eqnarray}
and for any constants $\delta>0$ and $0<\eta<1$,
\begin{eqnarray}\label {HI21}
P \left (2m\max_{1 \leq l < \infty} \lambda_l V^2_{n,l} \leq (1-\eta) n\sum_{1 \leq l < \infty}\lambda_l L_l(z_{n,l})\right )
\leq \exp (-2 x_n^2),
\end{eqnarray}
\begin{eqnarray}\label {I(2.9)}
&& P \left (\max_{1 \leq l <\infty} \lambda_l \sum_{i=1}^n g^2_l(X_i)I(|g_l(X_i)| \leq \delta z_{n,l}) \leq (1-\eta) n \max_{1 \leq l <\infty}\lambda_l L_l(z_{n,l})\right )  \nonumber\\
&\leq& \exp (-2 x_n^2).
\end{eqnarray}
In particular,
\begin{eqnarray}
I_{1,2}\le P \left (\max_{1 \leq l <\infty} \lambda_l V^2_{n,l} \leq (1-\varepsilon) n \max_{1 \leq l <\infty}\lambda_l L_l(z_{n,l})\right )  \leq \exp (-2 x_n^2). \nonumber
\end{eqnarray}
\end{proposition}
\begin{proof}
We shall apply the following exponential inequality (See, e.g., Theorem 2.19 of de la Pe\~{n}a, Lai and Shao  \cite{dLS}). If $Y_1, ..., Y_n$ are independent random variables with $Y_i \geq 0$, $\mu_n = \sum_{i=1}^n E Y_i$ and $B_n^2 = \sum_{i=1}^n E Y_i^2<\infty$, then for $0 < x < \mu_n$,
\begin{eqnarray}
P \bigg (\sum_{i=1}^n Y_i  \leq x \bigg ) \leq \exp \left ( - \frac{(\mu_n - x)^2}{2 B_n^2}  \right ). \nonumber
\end{eqnarray}
By (\ref{a1}), $\sum_{l=m+1}^{\infty}\lambda_l V^2_{n,l}/\sum_{l=1}^{\infty}\lambda_l V^2_{n,l} \rightarrow 0$ as $m \rightarrow \infty$. Then by (\ref{a1b}),
\begin{eqnarray}\label {a1a}
\frac{\sum_{l=m+1}^{\infty}\lambda_l V^2_{n,l}}{\max_{1 \leq l <\infty}\lambda_l V^2_{n,l}} \rightarrow 0  \ \ \  as \ \ m \rightarrow \ \infty.
\end{eqnarray}
Hence
\begin{eqnarray*}
\varepsilon  \max_{1 \leq l <\infty} \lambda_l V^2_{n,l } \geq 2   \sum_{l=m+1}^{\infty} \lambda_l V^2_{n,l}\geq 2   \sum_{l=m+1}^{\infty} \lambda_l\sum_{i=1}^n \bar{g}_l^2(X_i).
\end{eqnarray*}
Then
\begin{eqnarray}\label {A2.11}
I_{1,1}&\leq& P \left (\sum_{l=m+1}^{\infty} \lambda_l\sum_{i=1}^n \bar{g}_l^2(X_i)\leq  \frac{n}{2}  \sum_{l=m+1}^{\infty}\lambda_l L_l(z_{n,l})\right ) \nonumber\\
&\leq& \exp \left ( -\frac{  (n  \sum_{l=m+1}^{\infty}\lambda_{l} L_{l}(z_{n,l})/2)^2}{ 2 n   E(\sum_{l=m+1}^{\infty} \lambda_l \bar{g}^2_l(X_1))^2} \right ).
\end{eqnarray}
By Minkowski's integral inequality, (\ref{nznl}) and (\ref{Egpx}),
\begin{eqnarray}\label {A2.12}
E \left(\sum_{l=m+1}^{\infty} \lambda_l \bar{g}^2_l(X_1) \right)^2 &\leq& \left \{\sum_{l=m+1}^{\infty}\lambda_l  (E \bar{g}_l^4(X_1))^{1/2}\right\}^2\nonumber\\
& = &o\left(\frac{\sqrt{n}}{x_n} \sum_{l=m+1}^{\infty}\lambda_l L_l(z_{n,l})\right)^2.
\end{eqnarray}
Therefore, (\ref{I1(1)}) follows from (\ref{A2.11}) and (\ref{A2.12}).
To show \eqref{HI21}, notice that by \eqref{a1b},
\begin{eqnarray*}\label{a1b'}
m \max_{1 \leq l < \infty}\lambda_l V^2_{n,l}\geq (1-\eta)\sum_{l=1}^{\infty}\lambda_l V^2_{n,l}\geq  (1-\eta)  \sum_{l=1}^{\infty} \lambda_l\sum_{i=1}^n \bar{g}_l^2(X_i).
\end{eqnarray*}
Then \begin{eqnarray*}
&&P \left (2m\max_{1 \leq l < \infty} \lambda_l V^2_{n,l} \leq (1-\eta) n\sum_{1 \leq l < \infty}\lambda_l L_l(z_{n,l})\right ) \\
&\le& P \left (\sum_{l=1}^{\infty} \lambda_l\sum_{i=1}^n \bar{g}_l^2(X_i)\leq \frac{n}{2}  \sum_{l=1}^{\infty}\lambda_l L_l(z_{n,l})\right ) \nonumber\\
&\leq& \exp \left ( -\frac{  ((\frac{n}{2}  \sum_{l=1}^{\infty}\lambda_{l} L_{l}(z_{n,l}))^2}{ 2 n   E(\sum_{l=1}^{\infty} \lambda_l \bar{g}^2_l(X_1))^2} \right ).
%\leq \exp (-2 x_n^2),
\end{eqnarray*}
Similar to the proof of $I_{1,1}$ as in (\ref{A2.11}) and (\ref{A2.12}), we have \eqref{HI21}.

To show (\ref{I(2.9)}), let
\begin{eqnarray*}\label  {ln}
l_n = \min \left \{l: \lambda_l L_l(z_{n,l}) = \max_{1 \leq l < \infty}\lambda_l L_l(z_{n,l}) \right \}.
\end{eqnarray*}
Then
\begin{eqnarray*}\label {LIL2.7}
&& P \left  (\max_{1 \leq l <\infty} \lambda_{l}\sum_{i=1}^n g^2_{l}(X_i) I(|g_l(X_i)| \leq \delta z_{n,l})\leq (1-\eta)n \max_{1 \leq l <\infty} \lambda_{l} L_{l} (z_{n,l})    \right )\nonumber \\
&\leq& P \left  (  \sum_{i=1}^n \lambda_{l_n}g^2_{l_n}(X_i) I(|g_{l_n}(X_i)| \leq \delta z_{n,l_n})\leq (1-\eta)n   \lambda_{l_n} L_{l_n} (z_{n,l_n})    \right )\nonumber \\
&\leq& \exp \left ( -\frac{(1-(1- \eta))^2 (n  \lambda_{l_n} L_{l_n}(z_{n,l_n}))^2}{ 2 n \lambda^2_{l_n}E  g_{l_n}^4(X_1)I(|g_{l_n}(X_1)| \leq \delta z_{n,l_n})} \right )\nonumber\\
&=&\exp \left ( -\frac{ \eta^2 (n  \lambda_{l_n} L_{l_n}(z_{n,l_n}))^2}{ 2 n \lambda^2_{l_n}o(n L^2_{l_n}(\delta z_{n, l_n})/x_n^2)} \right ).
\end{eqnarray*}
Since $L_l(\delta z_{n,l})/L_l(z_{n,l}) \rightarrow 1$, then (\ref{I(2.9)}) follows.
\end{proof}

% -------------------------------------------------------------------------------------------------subsection
\subsection{Estimation of $I_2$}

%Let $[a ]$ denote the largest integer smaller than or equal to $a$. In addition to the estimate of $I_2$, we show (\ref{md2.10w}) in the following proposition which will be used in the proof of Theorem \ref{LIL}.

% ---------------------------------------------------------------------------------proposition
\begin{proposition}\label {proposition2}
\begin{eqnarray*}
I_2= P\bigg ( \frac{\sum_{l=1}^{\infty} \lambda_l \left \{  \sum_{i=1}^{n} g_l(X_i) I \left (|g_l(X_i)| > z_{n,l} \right )  \right \}^2}{  \max_{1 \leq l <\infty}\lambda_l V^2_{n,l}} \geq \varepsilon^2 x_n^2 \bigg ) \leq  \exp (- 2x_n^2). \nonumber
\end{eqnarray*}
\end{proposition}
\begin{proof}
By Cauchy-Schwarz inequality,
\begin{eqnarray*}
&& \frac{\sum_{l=1}^{\infty} \lambda_l \big \{  \sum_{i=1}^{n} |g_l(X_i)| I \left (|g_l(X_i)| > z_{n,l} \right )  \big \}^2}{\max_{1 \leq l <\infty}\lambda_l  V^2_{n,l}} \nonumber\\
&\leq&\frac{\sum_{l=1}^{\infty} \lambda_l  \sum_{i=1}^{n} g^2_l(X_i)\sum_{i=1}^{n}   I \left (|g_l(X_i)| > z_{n,l} \right )   }{\max_{1 \leq l <\infty}\lambda_l  V^2_{n,l}}.
\end{eqnarray*}
By (\ref{a1b}), the sum of the diagonal terms
\begin{eqnarray*}\label {md2.14z}
\frac{\sum_{l=1}^{\infty} \lambda_l  \sum_{i=1}^{n} g^2_l(X_i)  I \left (|g_l(X_i)| > z_{n,l} \right )   }{\max_{1 \leq l <\infty}\lambda_l  V^2_{n,l}} \leq   \frac{\varepsilon^2 x_n^2}{2}.
\end{eqnarray*}

\begin{eqnarray} \label {md2.15d}
I_2 &\leq& P\Bigg  ( \frac{\sum_{1 \leq i \neq j \leq n}\sum_{l=1}^{\infty} \lambda_l   g^2_l(X_i)  I   (|g_l(X_j)| >  z_{n,l}) }{\max_{1 \leq l <\infty}\lambda_l  \sum_{i=1}^{n} g^2_l(X_i)} \geq \frac{\varepsilon^2  x_n^2}{2} \Bigg )\nonumber\\
&\leq&P\Bigg  ( \frac{\sum_{1 \leq i < j \leq n}\sum_{l=1}^{\infty} \lambda_l   g^2_l(X_i)  I   (|g_l(X_j)| >  z_{n,l}) }{\max_{1 \leq l <\infty}\lambda_l  \sum_{i=1}^{n} g^2_l(X_i)} \geq \frac{\varepsilon^2  x_n^2}{4} \Bigg )\nonumber\\
&+&P\Bigg  ( \frac{\sum_{1 \leq j< i \leq n}\sum_{l=1}^{\infty} \lambda_l   g^2_l(X_i)  I   (|g_l(X_j)| >  z_{n,l}) }{\max_{1 \leq l <\infty}\lambda_l  \sum_{i=1}^{n} g^2_l(X_i)} \geq \frac{\varepsilon^2  x_n^2}{4} \Bigg )\nonumber\\
&\leq&P\Bigg  ( \sum_{2\le j\le n}\frac{\sum_{1 \leq i < j }\sum_{l=1}^{\infty} \lambda_l   g^2_l(X_i)  I   (|g_l(X_j)| >  z_{n,l}) }{\max_{1 \leq l <\infty}\lambda_l  \sum_{1 \leq i < j } g^2_l(X_i)} \geq \frac{\varepsilon^2  x_n^2}{4} \Bigg )\nonumber\\
&+&P\Bigg  ( \sum_{1 \leq j< n}\frac{\sum_{j< i \leq n}\sum_{l=1}^{\infty} \lambda_l   g^2_l(X_i)  I   (|g_l(X_j)| >  z_{n,l}) }{\max_{1 \leq l <\infty}\lambda_l  \sum_{j< i \leq n} g^2_l(X_i)} \geq \frac{\varepsilon^2  x_n^2}{4} \Bigg )\nonumber\\
&=&I_{2,1}+I_{2,2}.
\end{eqnarray}
Let
\begin{eqnarray*}
\phi_j = \frac{\sum_{1 \leq i < j}\sum_{l=1}^{\infty} \lambda_l   g^2_l(X_i)  I   (|g_l(X_j)| >  z_{n,l}) }{ \max_{1 \leq l <\infty}\lambda_l \sum_{1 \leq i < j } g^2_l(X_i)} .
\end{eqnarray*}
Then for any constant $t>0$,
\begin{eqnarray}\label {LI2.20b}
I_{2,1} \leq E e^{t \sum_{j=2}^{n} \phi_j}e^{-t \varepsilon^2 x_n^2/4}.
\end{eqnarray}
Let $E_j$ be the expectation of $X_j$ for $2 \leq j \leq n$. Then
\begin{eqnarray}  \label {LI2.21a}
E e^{t \sum_{j=2}^n \phi_j}=E (e^{t \sum_{j=2}^{n-1} \phi_j}E_{ n} e^{ t \phi_{n}}).
\end{eqnarray}
Since $|e^s-1| \leq e^{0 \vee s}|s|$ for any $s \in R$ and $0 \leq \phi_n \leq m$ for some $m$ sufficiently large, then
\begin{eqnarray*} \label {LI2.22b}
  \left |E_n e^{t \phi_{n}} - 1 \right | &\leq& e^{m t} t E_{n} \phi_{n} \nonumber\\
&=& \frac{e^{m t} t \sum_{1 \leq   i < n}\sum_{l=1}^{\infty} \lambda_l   g^2_l(X_i)   P   (|g_l(X_{n})| >  z_{n,l}) }{ \max_{1 \leq l <\infty}\lambda_l \sum_{ 1 \leq   i <n} g^2_l(X_i)}.\nonumber\\
\end{eqnarray*}
By (\ref{Pgxa}) and (\ref{nznl}), we have $P   (|g_l(X_{n})| >  z_{n,l}) =o(x_n^2/n)$. Then together with (\ref{a1b}),
\begin{eqnarray} \label {LI2.23a}
E_{n} e^{t   \phi_{n}} = 1 + o(x_n^2/n) = e^{o(x_n^2/n)}.
\end{eqnarray}
Applying (\ref{LI2.23a}) to (\ref{LI2.21a}), we have
\begin{eqnarray*}
E e^{t \sum_{j=2}^{n} \phi_j}= e^{o(x_n^2/n)} E e^{t \sum_{j=2}^{ n-1} \phi_j} .
\end{eqnarray*}
Similarly,
\begin{eqnarray*}
E e^{t \sum_{j=2}^{n-1} \phi_j}&=& E (e^{t \sum_{j=2}^{ n-2} \phi_j}E_{n-1} e^{t\phi_{n-1}})\nonumber\\
&=& e^{o(x_n^2/n)} E e^{t\sum_{j=2}^{n-2} \phi_j} .
\end{eqnarray*}
Continue this process from $X_{n}$ to $X_1$, we conclude
\begin{eqnarray}  \label {LI2.26b}
E e^{t \sum_{j=2}^{ n} \phi_j}= e^{ n \times o(x_n^2/n)} =e^{o(x^2_n)} .
\end{eqnarray}
Applying (\ref{LI2.26b}) to (\ref{LI2.20b}) and letting $t=16/\varepsilon^2$, we have
\begin{eqnarray}\label {LI2.29}
I_{2,1} \leq \exp (-3x_n^2).
\end{eqnarray}
By the same argument,
\begin{eqnarray} \label{LI2.34b}
I_{2,2} \leq \exp (-3x_n^2).
\end{eqnarray}
Combining (\ref{md2.15d}), (\ref{LI2.29}) and (\ref{LI2.34b}), we obtain the proposition.
\end{proof}

% ----------------------------------------------------------------------------------------- subsection

\subsection{Estimation of $I_3$}

Let $Y_1,..., Y_n$ be an independent copy of $X_1,...,X_n$. We will use the following lemma which is a Bernstein-type exponential inequality for degenerate {\it U}-statistics.

% ------------------------------------------------------------------------------ lemma

\begin{lemma} \label {GLZ}
((3.5) of Gin\a'{e},  Lata\l a and Zinn \cite{GLZ}). For bounded degenerate kernel $h_{i,j}(X_i, Y_j)$, let
\begin{eqnarray}
A=\max_{i,j} \left \| h_{i,j}(X_i, Y_j)   \right \|_{\infty}, \quad C^2 = \sum_{i,j} E h_{i,j}^2(X_i, Y_j), \qquad \qquad  \nonumber \\
B^2 = \max_{i,j} \bigg \{  \Big  \| \sum_i E h_{i,j}^2(X_i,y)  \Big  \|_{\infty}, \Big  \| \sum_j E h_{i,j}^2(x, Y_j)  \Big   \|_{\infty}  \bigg \}. \nonumber
\end{eqnarray}
Then there is a universal constant $K$ such that
\begin{eqnarray}
Pr \bigg \{\Big  |   \sum_{i,j} h_{i,j}(X_i, Y_j)   \Big  | >x     \bigg \}   \leq K \exp \left \{  -\frac{1}{K} \min \bigg [ \frac{x}{C},   \left ( \frac{x}{B}  \right )^{2/3}, \left ( \frac{x}{A}   \right )^{1/2} \bigg ]\right \}. \nonumber
\end{eqnarray}
\end{lemma}

% --------------------------------------------------------------------------------------------
Recall \eqref{barest}.  Hence by (\ref{XY}) and the definition of $I_3$ in (\ref{m2.10}),
\begin{eqnarray}\label{I2.36}
I_3 \leq P \bigg ( \frac{ \sum_{l=m+1}^{\infty} \lambda_l (\bar{Y}_{n,l}-E \bar{Y}_{n,l})^2}{\sum_{l=m+1}^{\infty}\lambda_l L_l(z_{n,l})} \geq C_1  n (1-2 \varepsilon)^2  x_n^2  \bigg ).
\end{eqnarray}
Let
\begin{eqnarray}\label {hm}
h_m (X_i, Y_j)= \sum_{l=m+1}^{\infty} \lambda_l (\bar{g}_l(X_i)-E \bar{g}_l(X_i))(\bar{g}_l(Y_j)-E \bar{g}_l(Y_j)).
\end{eqnarray}
In addition to the estimate of $I_{3}$, we include (\ref{AC2.33}) in the following proposition which will be used in the proof of Theorem \ref{LIL}, where
\begin{eqnarray} \label {hbeta}
&&h_{(\beta) }(X_i, Y_j) \nonumber\\
&=& \sum_{l=1}^{\infty}\lambda_l  \left \{\big (g_{l}(X_i)I(|g_l(X_i)| \leq \beta   z_{n,l})-E  g_{l}(X_i) I(|g_l(X_i)| \leq \beta   z_{n,l})\right\} \nonumber\\
&& \times  \left\{(g_{l}(Y_j)I(|g_l(Y_j)| \leq \beta  z_{n,l})-E g_{l}(Y_j)\big )I(|g_l(Y_j)| \leq \beta  z_{n,l}) \right \}.
\end{eqnarray}

%----------------------------------------------------------------------------------------------- lemma

\begin{proposition} \label {supfV}
For constant $C_2>0$ sufficiently large,
\begin{eqnarray} \label {2.35a}
 P\bigg (  \frac{  \sum_{1 \leq i , j \leq n} h_m(X_i, Y_j)}{ n \sum_{l=m+1}^{\infty}\lambda_l   L_l (z_{n, l}) } \geq  C_2  x_n^2 \bigg) \leq \exp (-3 x_n^2),
\end{eqnarray}
Then by (\ref{I2.36}) and the decoupling inequalities of de la Pe\~{n}a and Montgomery-Smith \cite{dMS95}, for $C_1>0$ sufficiently large,
\begin{eqnarray*} \label {2.35b}
I_{3} \leq P \bigg (  \frac{  \sum_{1 \leq i , j \leq n} h_m(X_i, X_j)}{ n \sum_{l=m+1}^{\infty}\lambda_l   L_l (z_{n, l}) } \geq  C_1  (1-2 \varepsilon)^2 x_n^2  \bigg) \leq \exp (-2 x_n^2).
\end{eqnarray*}
Suppose that $\lambda_l>0$ for all $1 \leq l <\infty$ and $d>0$ is a constant. For constants $0<\alpha, \beta \leq 1$ sufficiently small,
\begin{eqnarray}\label {AC2.33}
P \bigg (  \frac{ \sum_{1 \leq i , j \leq [\alpha n]} h_{(\beta)}(X_i, Y_j)}{  n \sum_{l=1}^{\infty} \lambda_l L_l ( z_{n,l}) } \geq  d x_n^2 \bigg )\leq \exp (-2 x^2_{n}) .
\end{eqnarray}
\end{proposition}
\begin{proof}
We will prove (\ref{2.35a}) and (\ref{AC2.33}) simultaneously. By (\ref{nznl}) and (\ref{hm}),
\begin{eqnarray}\label{A}
A_n& :=& \left \|h_m (X_i, Y_j) \right\|_{\infty} \leq 4\sum_{l=m+1}^{\infty} \lambda_l   z_{n,l}^2 = 4 \sum_{l=m+1}^{\infty} \lambda_l \frac{n L_l(z_{n,l})}{x_n^2}  .
\end{eqnarray}
 By (\ref{nznl}) and (\ref{hbeta}),
\begin{eqnarray}\label {Ad}
A_{n,(\beta) } := \left \|h_{(\beta) }(X_i, Y_j) \right\|_{\infty} \leq 4 \sum_{l=1}^{\infty}\lambda_l \beta^2 z^2_{n,l} = \frac{4 \beta^2 n \sum_{l=1}^{\infty}\lambda_lL_l(z_{n,l})}{x_n^2}.
\end{eqnarray}
Let
\begin{eqnarray}
B_n^2 := \max \bigg \{ \Big \| \sum_{1 \leq i \leq n} E h_m^2 (X_i, y)  \Big  \|_{\infty}, \ \Big  \| \sum_{1 \leq j \leq n} E h_m^2(x, Y_j)  \Big  \|_{\infty}  \bigg \}.  \nonumber
\end{eqnarray}
Since $|\bar{g}_l(Y_j)|\leq z_{n,l} $, then by Cauchy-Schwarz inequality, (\ref{nznl}) and (\ref{hm}),
\begin{eqnarray*}\label {AC2.38}
&& \Big  \| \sum_{1 \leq i \leq n } E h_m^2(X_i, y)\Big  \|_{\infty}\leq nE \bigg (2 \sum_{l=m+1}^{\infty}\lambda_l |\bar{g}_l(X_1) - E \bar{g}_l(X_1)|  z_{n,l} \bigg  )^2    \nonumber \\
&\leq & 4 n  E   \sum_{l=m+1}^{\infty}\lambda_l  \big(\bar{g}_l(X_1) - E \bar{g}_l(X_1)\big)^2   \sum_{l=m+1}^{\infty}\lambda_l z^2_{n,l}   \nonumber \\
&\leq& 4 n  \sum_{l=m+1}^{\infty} \lambda_l L_l(z_{n,l})  \sum_{l=m+1}^{\infty} \lambda_l \frac{n L_l(z_{n,l})}{x_n^2} .
\end{eqnarray*}
The same result can be obtained for $\| \sum_{1 \leq j \leq n } E h_m^2(x, Y_j)  \|_{\infty}$. Therefore
\begin{eqnarray} \label{B2}
B_n^2 \leq  \frac{4   n^2 (\sum_{l=m+1}^{\infty}\lambda_l L_l(z_{n,l}))^2}{x_n^2}.
\end{eqnarray}
Similarly,
\begin{eqnarray*} \label {Bd}
B_{n, \alpha, (\beta)}^{2} &:= &\max \bigg \{ \Big \| \sum_{1 \leq i \leq [\alpha  n]} E h_{(\beta)}^{2} (X_i, y)  \Big  \|_{\infty}, \Big \| \sum_{1 \leq j \leq [\alpha  n]} E h_{(\beta)}^{2}(x, Y_j)  \Big  \|_{\infty}  \bigg \} \nonumber \\
&\leq& 4 \alpha n \sum_{l=1}^{\infty} \lambda_l L_l(\beta z_{n,l})  \sum_{l=1}^{\infty} \lambda_l \beta^2 \frac{n L_l(z_{n,l})}{x_n^2}.
\end{eqnarray*}
Since $0 < \beta \leq 1$, then $L_l(\beta z_{n,l})/L_l(z_{n,l}) \leq 1$. Hence
\begin{eqnarray*}
B_{n, \alpha, (\beta)}^{2} \leq \frac{4   \alpha   \beta^2    n^2 (\sum_{l=1}^{\infty}\lambda_l L_l(z_{n,l}))^2}{x_n^2}.
\end{eqnarray*}
By (\ref{hm}) and the Cauchy-Schwarz inequality,
\begin{eqnarray} \label {C2a}
C_{n}^2 &:=&\sum_{1 \leq i ,  j \leq n} E h_m^2(X_i, Y_j) \nonumber\\
&\leq&  \sum_{1 \leq i, j \leq n}   \sum_{l=m+1}^{\infty} \lambda_ l E\big (\bar{g}_l(X_i) - E \bar{g}_l(X_i)\big )^2  \sum_{l=m+1}^{\infty} \lambda_l E\big (\bar{g}_l(Y_j) - E \bar{g}_l(Y_j)\big )^2  \nonumber \\
&\leq&  n^2 \Big (\sum_{l=m+1}^{\infty} \lambda_l L_l(z_{n,l})   \Big )^2 .
\end{eqnarray}
Similarly,
\begin{eqnarray}\label {Cd}
C_{n,\alpha, (\beta) }^2 &:=&  \sum_{1 \leq i \neq  j \leq [\alpha  n]} E h_{(\beta)}^2(X_i, Y_j) \nonumber\\
&\leq&    \alpha^2   n^2  \bigg(\sum_{l=1}^{\infty}\lambda_l L_l(\beta   z_{n,l}) \bigg)^2 \nonumber\\
&\leq&   \alpha^2    n^2  \bigg (\sum_{l=1}^{\infty}\lambda_l L_l( z_{n,l}) \bigg)^2 .
\end{eqnarray}
Now let
\begin{eqnarray}\label {x}
x=  C_2 n x_n^2 \sum_{l=m+1}^{\infty} \lambda_l L_l(z_{n,l}).
\end{eqnarray}
By (\ref{A}) and (\ref{x}),
\begin{eqnarray} \label {xA}
\left ( \frac{x}{A_n} \right )^{1/2}\geq\left (  \frac{  C_2   n  x_n^2  \sum_{l=m+1}^{\infty}  \lambda_l L_l(z_{n,l})}{4  n  \sum_{l=m+1}^{\infty} \lambda_l L_l(z_{n,l})/ x_n^2}   \right )^{1/2} =   \left ( C_2  /4\right )^{1/2} x_n^2. \nonumber
\end{eqnarray}
By (\ref{B2}) and (\ref{x}),
\begin{eqnarray}\label {xB}
\left (  \frac{x}{B_n} \right )^{2/3} \geq \left (  \frac{ C_2  n x_n^2  \sum_{l=m+1}^{\infty} \lambda_l L_l(z_{n,l})}{2  n  \sum_{l=m+1}^{\infty} \lambda_l L_l(z_{n,l})/ x_n} \right )^{2/3} = ( C_2   /2)^{2/3}x_n^2.  \nonumber
\end{eqnarray}
By (\ref{C2a}) and (\ref{x}),
\begin{eqnarray}\label {xC}
\frac{x}{C_n} \geq \frac{ C_2   n x_n^2   \sum_{l=m+1}^{\infty} \lambda_l L_l(z_{n,l})}{   n  \sum_{l=m+1}^{\infty} \lambda_l  L_l(z_{n,l})} =  C_2  x_n^2. \nonumber
\end{eqnarray}
Then (\ref{2.35a}) follows from Lemma \ref{GLZ} for sufficiently large $C_2$. Similarly, let
\begin{eqnarray}\label {xd}
x_{d}= d  n x^2_{n} \sum_{l=1}^{\infty}\lambda_l L_l(z_{n,l}) .
\end{eqnarray}
By (\ref{Ad}) and (\ref{xd}),
\begin{eqnarray}
\left ( \frac{x_{d}}{A_{n, (\beta)}} \right )^{1/2}\geq\left (  \frac{d   n  x^2_{n}  \sum_{l=1}^{\infty}\lambda_l  L_l(z_{n,l}) }{4   \beta^2  n \sum_{l=1}^{\infty} \lambda_l L_l(z_{n,l})/x_{n}^2}   \right )^{1/2} =   \frac{  \sqrt{d} }{2  \beta   }  x_{n}^2. \nonumber
\end{eqnarray}
By (\ref{Bd}) and (\ref{xd}),
\begin{eqnarray}
\left (  \frac{x_{d}}{B_{n, \alpha, (\beta)}} \right )^{2/3} \geq \left (  \frac{d n x^2_{n} \sum_{l=1}^{\infty}\lambda_l  L_l(z_{n,l})}{2\sqrt{ \alpha }    \beta n  \sum_{l=1}^{\infty}\lambda_l L_l(z_{n,l}) /x_{n}} \right )^{2/3} =\left( \frac{d }{2\sqrt{ \alpha }  \beta }\right)^{2/3}x_{n}^2. \nonumber
\end{eqnarray}
By (\ref{Cd}) and (\ref{xd}),
\begin{eqnarray}
\frac{x_{d}}{C_{n,\alpha, (\beta)}} \geq \frac{d n x^2_{n} \sum_{l=1}^{\infty} \lambda_l L_l(z_{n,l})}{   \alpha  n  \sum_{l=1}^{\infty}\lambda_l  L_l(z_{n,l})} =   \frac{d }{ \alpha  } x_{n}^2 . \nonumber
\end{eqnarray}
Therefore (\ref{AC2.33}) follows from Lemma \ref {GLZ} for $\alpha$ and $\beta$ sufficiently small.
\end{proof}
\bigskip

% ----------------------------------------------------------------------------------------------- subsecgtion

\subsection {Estimation of $I_4$}

%----------------------------------------------------------------------------------------------- lemma

%
%\begin{lemma} \label{CS}
%(Part of Lemma 3.1 of Cs\"{o}rg\H{o} and Shao \cite{CS}). Let $\{ \xi_l, l \geq 1  \}$ be independent normal random variables with $E \xi_l=0$ and $\sum_{l=1}^{\infty} (E \xi_l^2)^{p/2}<\infty$. Then for $p \geq 2$,
%\begin{eqnarray}
%P \left (\left| \Big (\sum_{l=1}^{\infty} |\xi_l|^p  \Big )^{1/p}- E \Big (\sum_{l=1}^{\infty} |\xi_l|^p  \Big )^{1/p} \right | \geq x \right ) \leq 2 \exp \left (  - \frac{x^2}{2 \max_{l \geq 1}E \xi_l^2} \right ) \nonumber
%\end{eqnarray}
%for every $x > 0$.
%\end{lemma}

% ------------------------------------------------------------------------------------------------- lemma
Lemma \ref{Einmahl} below fellows Corollary 1(b) of Einmahl \cite{Einmahl} and Lemma 4.2 of Lin and Liu \cite{LL}. However, we add the condition $\sum_{i=1}^{k_n} E\| \xi_{n,i} \|^2 \leq b^2_n$ and our result is in a form of exponential inequality for independent random vectors. We use the same positive constants $c_{17}, c_{20}$ and $c_{22}$ (depending only on the vector dimension $d$) in Einmahl \cite{Einmahl}.
\begin{lemma} \label {Einmahl}
Let $\xi_{n,1}, ..., \xi_{n,k_n}$ be independent random vectors with mean zero and values in $\mathbb{R}^d$ such that $\|\xi_{n,i}\| \leq A_n$ and $\sum_{i=1}^{k_n} E\| \xi_{n,i} \|^2 \leq b^2_n$, where $\|\cdot\|$ denotes the Euclidean norm. Let $S_n = \sum_{i=1}^{k_n} \xi_{n,i}$. Suppose that
\begin{eqnarray}\label {K4.1c}
 Cov (S_n) = B_n I_d
\end{eqnarray}
where $B_n >0$ and $I_d$ is a $d \times d$ identity matrix, and $\alpha_n$ is a positive sequence such that $ \alpha_n B_n^{1/2}\rightarrow \infty$ and
\begin{eqnarray}\label {K4.2d}
\alpha_n\sum_{i=1}^{k_n}  E \left \{ \| \xi_{n,i} \|^3 \exp \left (  \alpha_n  \|\xi_{n,i}\|\right ) \right \} \leq B_n.
\end{eqnarray}
Let
\begin{eqnarray}\label {betan}
\beta_n  = B_n^{-3/2} \sum_{i=1}^{k_n} E \left \{ \| \xi_{n,i} \|^3 \exp \left (  \alpha_n \|\xi_{n,i}\|\right ) \right \}.
\end{eqnarray}
Then for any $0 < \gamma <1$, there exists $n_{\gamma}$ such that for all $n \geq n_{\gamma}$,
\begin{eqnarray}\label {K4.2b}
 P \left (\|S_n\| \geq  x  \right) &\leq& \exp \left \{ c_{20} \beta_n \left ( c_{17}^3 \alpha_n^3 B_n^{3/2}+1 \right) \right \} \nonumber\\
&&\times \bigg \{   \exp \left ( - \frac{(1-\gamma)^6 x^2}{2 B_n}  \right )  + \exp \left ( - \frac{\gamma^3 (1-\gamma)^3 x^2}{2  c_{22}B_n \beta_n ^2 \log (1/\beta_n)} \right )  \bigg \} \nonumber \\
&& + 2d \exp \left ( - \frac{(1-\gamma)^2 c^2_{17}  \alpha^2_n B^2_n }{2 (d^2 b^2_n +d c_{17} A_n\alpha_n B_n)}   \right ) \nonumber
\end{eqnarray}
uniformly for $x \in [e_n B_n^{1/2}, c_{17} \alpha_n B_n]$, where $\{e_n  \}_{n \geq 1}$ can be any sequence with $e_n \rightarrow \infty$ and $e_n \leq c_{17}\alpha_n B_n^{1/2}$.
\end{lemma}
\begin{proof}
Let $\eta_{n,i}$, $1 \leq i \leq k_n$, be independent $N(0, \sigma^2 Cov(\xi_{n,i}))$ random vectors, which are independent of the $ \xi_{n,i}$'s, where
\begin{eqnarray}\label {K4.4d}
\sigma^2 =  c_{22} \beta_n ^2 \log (1/\beta_n).
\end{eqnarray}
By (\ref{K4.2d}) and (\ref{betan}), we have $\beta_n \leq \alpha_n^{-1}B_n^{-1/2} \rightarrow 0$ as $n \rightarrow \infty$. Hence $\sigma \rightarrow 0$ as $n \rightarrow \infty$. Let $p_n(y)$ be the probability density of $B_n^{-1/2}\sum_{i=1}^{k_n} (\xi_{n,i} + \eta_{n,i})$, and $\phi_{(1+\sigma^2)I_d}$ be the density of $N(0, (1+\sigma^2) I_d)$. By Corollary 1(b) in Einmahl \cite{Einmahl} (together with the Remark on page 32), for $\|y\| \leq c_{17} \alpha_n B_n^{1/2} $,
\begin{eqnarray} \label {K4.8c}
p_n(y) = \phi_{(1+\sigma^2)I_d} (y) \exp (T_n(y)) \;\;\;   \textrm{with }\;\;\; |T_n(y)| \leq c_{20} \beta_n (\|y\|^3+1).
\end{eqnarray}
For any $0<\gamma <1$ and $x \in [e_nB_n^{1/2}, c_{17} \alpha_n B_n]$,
\begin{eqnarray}\label {K4.6b}
&& P \left (\|S_n\| \geq  x  \right) \nonumber \\
&\leq&   P \bigg ( \bigg \|S_n +\sum_{i=1}^{k_n} \eta_{n,i} \bigg\| \geq (1-\gamma) x  \bigg)    +     P \bigg ( \bigg \|\sum_{i=1}^{k_n} \eta_{n,i}  \bigg \| \geq \gamma x\bigg )        \nonumber \\
&=& P \bigg ((1-\gamma)x \leq  \bigg \|S_n +\sum_{i=1}^{k_n} \eta_{n,i} \bigg\|  <   c_{17} \alpha_n B_n  \bigg) \nonumber \\
&&+ P \bigg ( \bigg \|S_n +\sum_{i=1}^{k_n} \eta_{n,i} \bigg\|  \geq   c_{17} \alpha_n B_n  \bigg)  +     P \bigg ( \bigg \|\sum_{i=1}^{k_n} \eta_{n,i}  \bigg  \| \geq \gamma x\bigg )    \nonumber \\
&\leq& P \bigg ((1- \gamma)x \leq  \bigg \|  S_n + \sum_{i=1}^{k_n} \eta_{n,i} \bigg \| <  c_{17} \alpha_n B_n \bigg ) \nonumber \\
&& + P \left (\|S_n\|  \geq  (1-\gamma)c_{17} \alpha_n B_n    \right) \nonumber \\
&&+  P \bigg ( \bigg \|\sum_{i=1}^{k_n} \eta_{n,i}  \bigg \| \geq \gamma  c_{17} \alpha_n B_n \bigg ) +  P \bigg ( \bigg \|\sum_{i=1}^{k_n} \eta_{n,i}  \bigg \| \geq \gamma x\bigg )  \nonumber \\
&&  \nonumber \\
&\leq& P \bigg ((1- \gamma)x \leq  \bigg \|  S_n + \sum_{i=1}^{k_n} \eta_{n,i} \bigg \| <  c_{17} \alpha_n B_n \bigg )  +  2 P \bigg ( \bigg \|\sum_{i=1}^{k_n} \eta_{n,i}  \bigg \| \geq \gamma x\bigg) \nonumber \\
&&+ P \left (\|S_n\|  \geq  (1-\gamma)c_{17} \alpha_n B_n  \right)  \nonumber\\
&:=& J_1 +J_2 +J_3.
\end{eqnarray}
Let $N$ denote a centered normal random vector with covariance matrix $I_d$. Then by (\ref{K4.8c}),
\begin{eqnarray} \label {K4.9b}
&&J_1= \int_{ (1-\gamma)x/B_n^{1/2} < \|y\| \leq c_{17}\alpha_n B_n^{1/2}} \phi_{(1+\sigma^2)I_d} (y) \exp (T_n(y))dy \nonumber \\
&&\leq \exp \left \{ c_{20} \beta_n \left ( c_{17}^3 \alpha_n^3 B_n^{3/2}+1 \right )\right \} \int_{\|y\| \geq (1-\gamma)x/B_n^{1/2}} \phi_{(1+\sigma^2)I_d} (y) dy \nonumber \\
&&\leq  \exp \left \{  c_{20} \beta_n \left (  c^3_{17} \alpha_n^3 B_n^{3/2} +1 \right )  \right \}  \times\nonumber \\
&& \left \{ P \left (  \|N\| \geq (1-\gamma)^2x/B_n^{1/2} \right ) + P \left ( \sigma \|N\| \geq \gamma (1-\gamma) x/B_n^{1/2}  \right ) \right \}.
\end{eqnarray}
Observe that $\|N\|^2$ has $\chi_d^2$ distribution. It is well-known that for a  $\chi_d^2$ random variable $Y$, $P(Y>y)\le (ye^{1-y/d}/d)^{d/2}$ for $y>d$. Hence
\begin{align} \label {PNx}
& P \left (  \|N\| \geq  (1-\gamma)^2x/B_n^{1/2} \right ) \nonumber \\
& \leq \left(\frac{(1-\gamma)^4 x^2}{d B_n} \exp \left(1 - \frac{(1-\gamma)^4 x^2}{d B_n}  \right)\right)^{d/2}\nonumber \\
&=\frac{(1-\gamma)^{2d} x^d}{d^{d/2} B_n^{d/2}} \exp \left(d/2- \frac{(1-\gamma)^4 x^2}{2 B_n}  \right) \nonumber \\
&\leq \exp \bigg( - \frac{(1-\gamma)^6 x^2}{2 B_n}  \bigg )
\end{align}
by $x^2/B_n\rightarrow \infty$.
Similarly,
\begin{eqnarray} \label {K4.11e}
&& P\left (  \sigma \|N\| \geq \gamma (1-\gamma)x/B_n^{1/2} \right )  \leq \exp \left ( - \frac{\gamma^3 (1-\gamma)^3 x^2}{2 B_n \sigma^2} \right ) \nonumber \\
&=& \exp \left ( - \frac{\gamma^3 (1-\gamma)^3 x^2}{2 c_{22}B_n \beta_n ^2 \log (1/\beta_n)} \right )
\end{eqnarray}
 by (\ref{K4.4d}). Then by (\ref{K4.9b})-(\ref{K4.11e}), we have
\begin{eqnarray} \label {K4.13h}
 J_1&\leq& \exp \left \{ c_{20} \beta_n \left ( c_{17}^3 \alpha_n^3 B_n^{3/2}+1 \right) \right \} \nonumber \\
&& \times  \bigg \{   \exp \left ( - \frac{(1-\gamma)^6 x^2}{2 B_n}  \right )  + \exp \left ( - \frac{\gamma^3 (1-\gamma)^3 x^2}{2  c_{22}B_n \beta_n ^2 \log (1/\beta_n)} \right )  \bigg \}. \ \ \
\end{eqnarray}
Since the distribution of $\eta_{n,i}$ is $N(0, \sigma^2 Cov(\xi_{n,i}))$, then the distribution of $\sum_{i=1}^{k_n} \eta_{n,i}$ is $N(0, \sigma^2 \sum_{i=1}^{k_n} Cov(\xi_{n,i}))$. Since the $\xi_{n,i}$'s are independent, then
$ \sum_{i=1}^{k_n} Cov(\xi_{n,i})= Cov(\sum_{i=1}^{k_n} \xi_{n,i})=B_n I_d$ by (\ref{K4.1c}). Hence the distribution of $\sum_{i=1}^{k_n}\eta_{n,i}$ is $N(0, \sigma^2 B_n I_d)$. Then similar to (\ref{K4.11e}),
\begin{eqnarray}\label {K4.14h}
J_2 &=& 2P \bigg ( \bigg \|\sum_{i=1}^{k_n} \eta_{n,i}  \bigg \| \geq \gamma x\bigg )= 2 P \Big(  \sigma \|N\| \geq \gamma x/B_n^{1/2} \Big ) \ \nonumber \\
&\leq& 2   \exp \left ( - \frac{\gamma^3 x^2}{2 B_n c_{22} \beta_n^2 \log (1/\beta_n)  } \right ).
\end{eqnarray}
By (\ref{K4.13h}) and (\ref{K4.14h}), we have
\begin{eqnarray} \label {I1I2}
&&J_1 +J_2 \leq \exp \left \{ c_{20} \beta_n \left ( c_{17}^3 \alpha_n^3 B_n^{3/2} +1 \right) \right \} \nonumber \\
&& \times \bigg \{   \exp \left ( - \frac{(1-\gamma)^6 x^2}{2 B_n}  \right )  + 3\exp \left ( - \frac{\gamma^3 (1-\gamma)^3 x^2}{2  c_{22}B_n \beta_n ^2 \log (1/\beta_n)} \right )  \bigg \}.
\end{eqnarray}
Next we estimate $J_3$. For each $1 \leq i \leq n$, let $\xi_{n,i} = (\xi_{n,i}^{(1)}, ..., \xi_{n,i}^{(d)})^T $ where $\boldsymbol{a}^T$ denote the transpose of a vector $\boldsymbol{a}$. Then
\begin{eqnarray}\label {K4.14p}
\left \|  S_n \right \| = \bigg \|\sum_{i=1}^{k_n} \xi_{n,i} \bigg \|= \bigg (  \sum_{l=1}^d \Big (  \sum_{i=1}^{k_n} \xi^{(l)}_{n,i} \Big )^2 \bigg )^{1/2} \leq \sum_{l=1}^d \bigg |\sum_{i=1}^{k_n} \xi^{(l)}_{n,i}   \bigg |. \nonumber
\end{eqnarray}
Hence
\begin{eqnarray}
J_3 &=& P \left (\|S_n\|  \geq  (1-\gamma) c_{17} \alpha_n B_n   \right)  \nonumber \\
&\leq& P \bigg (\sum_{l=1}^d \bigg |\sum_{i=1}^{k_n} \xi_{n,i}^{(l)}   \bigg |  \geq  (1-\gamma) c_{17} \alpha_n B_n   \bigg)  \nonumber \\
&\leq& \sum_{l=1}^d  P \bigg ( \bigg |\sum_{i=1}^{k_n} \xi^{(l)}_{n,i}   \bigg |  \geq  \frac{(1-\gamma) c_{17} \alpha_n B_n }{d}   \bigg). \nonumber
\end{eqnarray}
Since $\|\xi_{n,i}\| \leq A_n$ and $\sum_{i=1}^{k_n}E \| \xi_{n,i} \|^2 \leq b^2_n$, then $|\xi^{(l)}_{n,i}| \leq \|\xi_{n,i}\| \leq A_n$ and $\sum_{i=1}^{k_n}E (\xi^{(l)}_{n,i})^2 \leq \sum_{i=1}^{k_n}E \|\xi_{n,i} \|^2 \leq b_n^2$ for each $1 \leq l \leq d$. By Bernstein's inequality (e.g., (2.17) of de la Pe\~{n}a, Lai and Shao \cite{dLS}),
\begin{eqnarray}\label {K4.15h}
J_3 &\leq&  2 d  \exp \left ( - \frac{(1-\gamma)^2c^2_{17} \alpha^2_n B^2_n }{2 d^2 (b^2_n + c_{17} A_n\alpha_n B_n/d)}   \right ).
\end{eqnarray}
Then the lemma follows by applying (\ref{I1I2}) and (\ref{K4.15h}) to (\ref{K4.6b}).
\end{proof}
\bigskip

% -------------------------------------------------------------------------------------------- lemma
Now we estimate $I_{4}$ in the following proposition which uses some ideas in Liu and Shao \cite{LS}.

\begin{proposition} \label {Lsup2}
\begin{eqnarray} \label {2.83}
I_{4}&\le &P \bigg ( \sum_{l=1}^{m} \lambda_l \left (\bar{Y}_{n,l} - E \bar{Y}_{n,l}   \right )^2\geq  (1-C_1\varepsilon)(1-2 \varepsilon)^3 n x_n^2 \max_{1\le l \le m} \lambda_l L_l (z_{n, l}) \bigg )\nonumber \\
&\leq& \exp \bigg ( -  \frac{  (1-C_1\varepsilon) (1-2\varepsilon)^4 x_n^2}{2(1+ \varepsilon)}   \bigg ). \nonumber
\end{eqnarray}
\end{proposition}
\begin{proof}
For each $1 \leq i \leq n$, let
\begin{eqnarray*} \label {Gi}
\boldsymbol{G}_{n,i} = \Big (\sqrt{\lambda_1} (\bar{g}_1(X_i) - E\bar{g}_1(X_i) ), ..., \sqrt{\lambda_m}(\bar{g}_m (X_i) - E\bar{g}_m(X_i)) \Big )^T.
\end{eqnarray*}
Let $B_n=n$ and $\Sigma$ be the covariance matrix of $\boldsymbol{G}_{n,1}$. For $1 \leq i \leq n$, let
\begin{eqnarray}\label {K4.7a}
\xi_{n,i} = \Sigma^{-1/2} \boldsymbol{G}_{n,i}.
\end{eqnarray}
Then
\begin{eqnarray}
&& Cov \left ( \xi_{n,1}+ \cdots + \xi_{n,n}  \right ) \nonumber \\
&=& E \left \{\left (  \Sigma^{-1/2} (\boldsymbol{G}_{n,1} + \cdots + \boldsymbol{G}_{n,n})\right )  \left (  \Sigma^{-1/2} (\boldsymbol{G}_{n,1} + \cdots + \boldsymbol{G}_{n,n})\right )^T \right \} \nonumber\\
&=&   \Sigma^{-1/2} E \left \{\left (\boldsymbol{G}_{n,1} + \cdots + \boldsymbol{G}_{n,n})  (\boldsymbol{G}_{n,1} + \cdots + \boldsymbol{G}_{n,n}\right )^T \right \}\Sigma^{-1/2} \nonumber \\
&=& \Sigma^{-1/2} \sum_{1 \leq i , j \leq n} E \left \{ \boldsymbol{G}_{n,i} \boldsymbol{G}_{n,j}^T \right \} \Sigma^{-1/2}. \nonumber
\end{eqnarray}
Since the $X_i$'s are independent, then
\begin{eqnarray*}
Cov \left ( \xi_{n,1}+ \cdots + \xi_{n,n}  \right ) &=&  \Sigma^{-1/2} \sum_{i=1}^n E \left \{\boldsymbol{G}_{n,i} \boldsymbol{G}_{n,i}^T \right \}  \Sigma^{-1/2} \nonumber \\
& =& n I_m = B_n I_m.
\end{eqnarray*}
Hence condition (\ref{K4.1c}) in Lemma \ref{Einmahl} is satisfied. Let
\begin{eqnarray*}\label {alphan}
\alpha_n = \frac{C_m x_n}{n^{1/2}}
\end{eqnarray*}
where $C_m>0$ is a finite constant depending only on $m$.  We shall verify condition (\ref{K4.2d}). By (\ref{K4.7a}),
\begin{eqnarray} \label {K4.8a}
\| \xi_{n,i} \|^2 =  \left (  \Sigma^{-1/2} \boldsymbol{G}_{n,i} \right )^T \left (  \Sigma^{-1/2} \boldsymbol{G}_{n,i} \right ) =  \boldsymbol{G}_{n,i}^T \Sigma^{-1} \boldsymbol{G}_{n,i}.
\end{eqnarray}
Observe that $\Sigma$ is positive definite by assumption (\ref{a2}). Then by the identity
\begin{eqnarray} \label {identity}
\boldsymbol{x^T A ^{-1}x}=\max_{\| \vartheta\|=1} \frac{(x^T \vartheta)^2}{\vartheta^T \boldsymbol{A} \vartheta}
\end{eqnarray}
for any $m \times m$ postive definite matrix $\boldsymbol{A}$, we have
\begin{eqnarray} \label {K4.10a}
\| \xi_{n,i} \|^2 = \boldsymbol{G}_{n,i}^T \Sigma^{-1} \boldsymbol{G}_{n,i} = \max_{\|  \vartheta \|=1} \frac{(\boldsymbol{G}_{n,i}^T \vartheta )^2}{\vartheta^T \Sigma \vartheta}.
\end{eqnarray}
Let $\vartheta^* =(\vartheta^*_1, ..., \vartheta^*_m)$ such that $\|\vartheta^*  \|=1$ and $(\boldsymbol{G}_{n,i}^T \vartheta^* )^2= \max_{\|\vartheta \|=1}(\boldsymbol{G}_{n,i}^T \vartheta )^2$. Then for any $\vartheta =(\vartheta_1, ..., \vartheta_m) \in l^2$, by the Cauchy-Schwarz inequality,
\begin{eqnarray} \label {K4.11a}
(\boldsymbol{G}_{n,i}^T \vartheta )^2 &=& \bigg ( \sum_{l=1}^m  \sqrt{\lambda_l} (\bar{g}_l(X_i) - E\bar{g}_l(X_i) ) \vartheta_l  \bigg  )^2 \nonumber \\
 &=& \bigg ( \sum_{l=1}^m  \frac{ \bar{g}_l(X_i) - E\bar{g}_l(X_i) }{\sqrt{L_l(z_{n,l})}} \vartheta_l \sqrt{\lambda_l L_l(z_{n,l})} \bigg )^2 \nonumber \\
&\leq&  \sum_{l=1}^m  \frac{(\bar{g}_l(X_i) - E\bar{g}_l(X_i))^2}{L_l(z_{n,l})}\sum_{l=1}^m  \vartheta^2_l \lambda_l L_l(z_{n,l}) .
\end{eqnarray}
Since $E g_l(X_1)=0$ for all $l \geq 1$, then $E \bar{g}_l(X_1)=o(x_n \sqrt{L_l(z_{n,l})}/\sqrt{n})$ by (\ref{Egx}) and (\ref{nznl}). By assumption (\ref{a2}),
\begin{eqnarray*} \label {K4.13ab}
\vartheta^T \Sigma \vartheta &=& \sum_{1 \leq l , l' \leq m} \vartheta_l \vartheta_{l'}\sqrt{\lambda_l \lambda_{l'}} E \left(\bar{g}_l(X_1) - E\bar{g}_l(X_1) \right ) \left(\bar{g}_{l'}(X_1) - E\bar{g}_{l'}(X_1) \right ) \nonumber \\
&=&  \sum_{l=1}^m \vartheta_l^2 \lambda_l  E \bar{g}^2_l(X_1) -  \sum_{l=1}^m \vartheta_l^2 \lambda_l  (E \bar{g}_l(X_1))^2 \nonumber \\
&& + \sum_{1 \leq l \neq  l' \leq m} \vartheta_l \vartheta_{l'}\sqrt{\lambda_l \lambda_{l'}} E \bar{g}_l(X_1) \bar{g}_{l'}(X_1)   \nonumber \\
&& - \sum_{1 \leq l \neq  l' \leq m} \vartheta_l \vartheta_{l'}\sqrt{\lambda_l \lambda_{l'}}  E\bar{g}_l(X_1) E\bar{g}_{l'}(X_1) \nonumber \\
&=& \sum_{l=1}^m  \vartheta^2_l \lambda_l L_l(z_{n,l}) -\sum_{l=1}^m \vartheta_l^2 \lambda_l  \times o\left ( \frac{x_n^2 L_l(z_{n,l})}{n} \right ) \nonumber \\
&& +o(1)\sum_{1 \leq l \neq l' \leq m} \vartheta_l \vartheta_{l'} \sqrt{\lambda_l \lambda_{l'}L_l(z_{n,l})L_{l'}(z_{n,l'})} \nonumber\\
&& -  \sum_{1 \leq l \neq l' \leq m} \vartheta_l \vartheta_{l'} \sqrt{\lambda_l \lambda_{l'}} \times o\bigg ( \frac{x_n \sqrt{ L_l(z_{n,l})}}{\sqrt{n}} \bigg ) o\bigg ( \frac{x_n \sqrt{ L_{l'}(z_{n,l'})}}{\sqrt{n}} \bigg ). \nonumber \\
\end{eqnarray*}
By the Cauchy-Schwarz inequality,
\begin{eqnarray}\label {J4.31a}
 \sum_{1 \leq l \neq l' \leq m} \vartheta_l \vartheta_{l'} \sqrt{\lambda_l \lambda_{l'}L_l(z_{n,l})L_{l'}(z_{n,l'})}\leq m \sum_{l=1}^m \vartheta_l^2 \lambda_l L_l(z_{n,l}). \nonumber
\end{eqnarray}
Hence
\begin{eqnarray} \label {K4.13a}
\vartheta^T \Sigma \vartheta  = (1+o(1)) \sum_{l=1}^m  \vartheta^2_l \lambda_l L_l(z_{n,l})  .
\end{eqnarray}
Applying (\ref{K4.11a}) and (\ref{K4.13a}) to (\ref{K4.10a}), we have
\begin{eqnarray} \label {K4.29f}
\| \xi_{n,i} \|^2  \leq 2 \sum_{l=1}^m  \frac{(\bar{g}_l(X_i) - E\bar{g}_l(X_i))^2}{L_l(z_{n,l})}.
\end{eqnarray}
Since $|\bar{g}_l(X_i)| \leq z_{n,l} = \sqrt{nL_{l}(z_{n,l})}/x_n$ by (\ref{nznl}), and \eqref{barest}, we have
\begin{eqnarray}\label {K4.15a}
\| \xi_{n,i} \|^2 \leq \frac{4 m n}{x_n^2} .
\end{eqnarray}
By (\ref{K4.29f}),
\begin{eqnarray}\label {K4.31h}
E \|\xi_{n,i} \|^2 \leq 2  \sum_{l=1}^m  \frac{E(\bar{g}_l(X_i) - E\bar{g}_l(X_i))^2}{L_l(z_{n,l})}  \leq 2m
\end{eqnarray}
and
\begin{eqnarray} \label {K4.20a}
E \| \xi_{n,i} \|^3 &\leq& 2^{3/2} E \left (\sum_{l=1}^m\frac{  (\bar{g}_l(X_i) - E\bar{g}_l(X_i))^2}{L_l(z_{n,l})} \right )^{3/2}.
\end{eqnarray}
By H\"{o}lder's inequality,
\begin{eqnarray}\label {A2.57b}
 \bigg (\sum_{l=1}^m\frac{  (\bar{g}_l(X_i) - E\bar{g}_l(X_i))^2}{L_l(z_{n,l})} \bigg )^{3/2}\leq m^{1/2}  \sum_{l=1}^m \frac{ |\bar{g}_l(X_i) - E\bar{g}_l(X_i)|^3}{L_l^{3/2}(z_{n,l})}.
\end{eqnarray}
Combining (\ref{K4.20a}) and (\ref{A2.57b}), we have
\begin{eqnarray}\label {K4.32f}
E \| \xi_{n,i} \|^3&\leq&  2^{3/2} m^{1/2}\sum_{l=1}^m \frac{8  E|\bar{g}_l(X_i) |^3}{L_l^{3/2}(z_{n,l})} \nonumber \\
&=& \sum_{l=1}^m  \frac{o(z_{n,l} L_l(z_{n,l}))}{L_l^{3/2}(z_{n,l})} = o \left ( \frac{n^{1/2} }{x_n} \right )
\end{eqnarray}
by (\ref{Egpx}) and (\ref{nznl}). Since $B_n=n$ and $\alpha_n = C_m x_n/n^{1/2}$, then by (\ref{K4.15a}) and (\ref{K4.32f}), we have
\begin{eqnarray}\label {K4.34f}
&& \alpha_n \sum_{i=1}^n  E \left \{ \| \xi_{n,i} \|^3 \exp \left (  \alpha_n \|\xi_{n,i}\|\right ) \right \} \nonumber \\
&\leq&  \frac{C_m x_n}{n^{1/2}}  n\times  o \left ( \frac{n^{1/2} }{x_n} \right ) \exp \bigg ( \frac{C_m x_n}{n^{1/2}} \left (  \frac{4 m n }{x_n^2} \right )^{1/2} \bigg ) = o(n) = o(B_n). \nonumber
\end{eqnarray}
Hence condition (\ref{K4.2d}) in Lemma \ref{Einmahl} is satisfied. Similarly,
\begin{eqnarray}\label {K4.37h}
\beta_n &:=& B_n^{-3/2}  \sum_{i=1}^n E \left \{ \| \xi_{n,i} \|^3 \exp \left (  \alpha_n \|\xi_{n,i}\|\right ) \right \} \nonumber \\
&=& n^{-3/2} n  \times o \left ( \frac{n^{1/2} }{x_n} \right ) \exp \left (\frac{C_m x_n}{n^{1/2}}  \bigg (  \frac{4 m n }{x_n^2} \right )^{1/2} \bigg ) \nonumber \\
&=& o(1/x_n).
\end{eqnarray}
Then $\beta_n^2 \log (1/\beta_n)=o(1/x_n)$. By (\ref{K4.15a}), we have $\|\xi_{n,i}\| \leq (4 m n /x_n^2)^{1/2}:=A_n$. By (\ref{K4.31h}), we have $\sum_{i=1}^n E \|\xi_{n,i} \|^2\leq 2mn:=b_n^2$. Then by Lemma \ref{Einmahl} and (\ref{K4.37h}) with $B_n=n$ and $\alpha_n = C_m x_n/n^{1/2}$ for sufficiently large $C_m$, we have
\begin{eqnarray}\label {K4.36g}
&& P \left (\|S_n\| \geq  n^{1/2}x_n  \right) \nonumber \\
&\leq& \exp \left \{ o(x_n^2) \right \}  \bigg \{   \exp \left ( - \frac{(1-\gamma)^6  n x_n^2}{2 n }  \right )+ \exp \bigg ( - \frac{\gamma^3 (1-\gamma)^3  n x_n^2}{n \times o(1/x_n)}  \bigg ) \bigg \} \nonumber\\
&& +  2m \exp \left ( - \frac{(1-\gamma)^2 c^2_{17} C_m^2 x_n^2  n }{2 (2 m^3 n +m(4 m n /x_n^2)^{1/2} c_{17} C_m x_n n^{1/2})}   \right ) \nonumber \\
&\leq& \exp \left \{ o(x_n^2) \right \}  \left \{   \exp \left ( - \frac{ (1-\gamma)^6  x^2_n}{2}  \right ) + \exp \left ( -4 x_n^2  \right ) \right \} + \exp \left ( -4 x_n^2  \right ) \nonumber \\
&\leq& \exp \left ( -  \frac{(1-\gamma)^7  x^2_n }{2} \right ).\nonumber
\end{eqnarray}
Letting $\gamma=1-(1-\varepsilon)^{1/7}$, we have
\begin{eqnarray} \label {K4.38k}
P \left (\|S_n\| \geq  n^{1/2}x_n  \right) \leq \exp \left (- \frac{(1-\varepsilon) x_n^2}{2} \right ).
\end{eqnarray}
Similar to (\ref{K4.8a}),
\begin{eqnarray} \label {K4.38g}
\|S_n\|^2= \Big \| \sum_{i=1}^n\xi_{n,i}\Big \|^2 &=&  \Big  (  \Sigma^{-1/2} \sum_{i=1}^n\boldsymbol{G}_{n,i} \Big )^T \Big (  \Sigma^{-1/2} \sum_{i=1}^n\boldsymbol{G}_{n,i} \Big  ) \nonumber \\
&=& \Big ( \sum_{i=1}^n \boldsymbol{G}_{n,i} \Big )^T \Sigma^{-1}  \Big ( \sum_{i=1}^n \boldsymbol{G}_{n,i} \Big ).
\end{eqnarray}
We will use the identity (\ref{identity}) to estimate (\ref{K4.38g}). Let $\vartheta^* = (\vartheta^*_1, ..., \vartheta^*_m)$ be such that $\|\vartheta^*\|=1$ and
\begin{eqnarray}\label {K4.40f}
\Big ( \sum_{i=1}^n \boldsymbol{G}_{n,i} \Big )^T\vartheta^*= \max_{\|\vartheta\|=1}  \Big ( \sum_{i=1}^n \boldsymbol{G}_{n,i}  \Big)^T \vartheta.
\end{eqnarray}
Observe that
\begin{eqnarray}
\max_{\|\vartheta\|=1}  \Big ( \sum_{i=1}^n \boldsymbol{G}_{n,i}  \Big)^T \vartheta = \bigg (\sum_{l=1}^m \Big ( \sum_{i=1}^n \sqrt{\lambda_l}(\bar{g}_l(X_i)- E \bar{g}_l(X_i))  \Big )^2 \bigg)^{1/2}.
\end{eqnarray}
By (\ref{K4.13a}),
\begin{eqnarray}\label {K4.41f}
(\vartheta^*)^T \Sigma \vartheta^* &=& (1+o(1)) \sum_{l=1}^m (\vartheta^*_l)^2 \lambda_l L_l(z_{n,l}) \nonumber \\
&\leq& (1+o(1))\max_{1 \leq l \leq m}\lambda_l L_l(z_{n,l})
\end{eqnarray}
because $\|\vartheta^*\|=1$. By the identity (\ref{identity}) and by (\ref{K4.40f})-(\ref{K4.41f}),
\begin{eqnarray}\label {K4.42g}
\Big ( \sum_{i=1}^n \boldsymbol{G}_{n,i} \Big )^T \Sigma^{-1}  \Big ( \sum_{i=1}^n \boldsymbol{G}_{n,i} \Big )\geq \frac{\sum_{l=1}^m \left ( \sum_{i=1}^n \sqrt{\lambda_l}(\bar{g}_l(X_i)- E \bar{g}_l(X_i))  \right )^2 }{(1+\varepsilon)\max_{1 \leq l \le m}\lambda_l L_l(z_{n,l})}. \nonumber \\
\end{eqnarray}
By (\ref{K4.38k}), (\ref{K4.38g}) and (\ref{K4.42g}), with application of \eqref{barest} and \eqref{XY},
\begin{eqnarray}\label {K4.43g}
I_{4}&\le & P \bigg ( \frac{\sum_{l=1}^m \lambda_l \left ( \sum_{i=1}^n (\bar{g}_l(X_i)- E \bar{g}_l(X_i))  \right )^2 }{\max_{1 \leq l \leq m}\lambda_l L_l(z_{n,l})} \geq  (1-C_1\varepsilon)(1-2 \varepsilon)^3  nx_n^2  \bigg ) \nonumber \\
&\leq& P \bigg (  \Big ( \sum_{i=1}^n \boldsymbol{G}_{n,i} \Big )^T \Sigma^{-1}  \Big ( \sum_{i=1}^n \boldsymbol{G}_{n,i} \Big ) \geq    \frac{ (1-C_1\varepsilon)(1-2 \varepsilon)^3  nx_n^2 }{1+ \varepsilon} \bigg ) \nonumber \\
&=& P \left (  \|S_n\|^2 \geq    \frac{ (1-C_1\varepsilon) (1-2 \varepsilon)^3 nx_n^2}{1+\varepsilon}  \right )\nonumber\\
& \le& \exp \left ( -  \frac{ (1-C_1\varepsilon)  (1-2\varepsilon)^4x_n^2}{2(1+\varepsilon)}   \right ). \nonumber
\end{eqnarray}
\end{proof}
\bigskip

Since $\varepsilon$ is arbitrary, then the upper bound of Theorem \ref{moderate} follows from (\ref{m2.10}) and the estimates of $I_1$, $I_2$, $I_3$ and $I_4$.

% ------------------------------------------------------------------------------------------------- section

\section {The lower bound of Theorem \ref{moderate}}

% --------------------------------------------------------------------------------------------- proof of the lower bound
\noindent
Let $0<\varepsilon<1$ be sufficiently small. For $1 \leq m < \infty$ sufficiently large, By (\ref{a1a}), $\max_{1 \leq l < \infty}\lambda_l V^2_{n,l}=\max_{1 \leq l \leq m}\lambda_l V^2_{n,l} $. Together with (\ref{a1b}), we have
\begin{eqnarray}\label {3.1}
 &&P (W_n \geq (1-\varepsilon) x_n^2) \nonumber\\
 &=& P \left ( \frac{\sum_{l=1}^{\infty}\lambda_l \left (\left \{\sum_{i=1}^n g_l(X_i) \right \}^2  - \sum_{i=1}^n g_l^2(X_i)\right) }{\max_{1 \leq l < \infty} \lambda_l V^2_{n,l}} \geq (1-\varepsilon)x_n^2 \right) \nonumber\\
&\geq& P \left ( \frac{\sum_{l=1}^{m}\lambda_l  \left \{\sum_{i=1}^n g_l(X_i) \right \}^2   }{\max_{1 \leq l \leq m} \lambda_l V^2_{n,l}} \geq x_n^2 \right).
\end{eqnarray}

Let $\tilde{g}_l(X_i)$ be the random variable with distribution which is of the distribution of $g_l(X_i)$ conditioned on $|g_l(X_i)|\le z_{n,l}$. Define $\tilde{Y}_{n,l}=\sum_{i=1}^n\tilde{g}_l(X_i)$ and $\tilde{V}_{n,l}^2=\sum_{i=1}^n\tilde{g}_l^2(X_i)$. By the definition of $L_l(x)$ and \eqref {Pgxa}
\begin{align*}
&E\tilde{g}_l^2(X_i)=E\bar{g}_l^2(X_i)/P(|g_l(X_i)|\le z_{n,l})\\
&=L_l(z_{n,l})/P(|g_l(X_i)|\le z_{n,l})=L_l(z_{n,l})(1+o(1)).
\end{align*}
Notice that \eqref{Egx} implies $E\tilde{g}_{l}(X_1)=o(L_l(z_{n,l})/z_{n,l})$. Then we have
\begin{align}\label{sigmal}
&\sigma_l^2:=E(\tilde{Y}_{n,l}-E\tilde{Y}_{n,l})^2=nE(\tilde{g}_{l}(X_1)-E\tilde{g}_{l}(X_1))^2\notag\\
&=nE\tilde{g}_{l}^2(X_1)(1+o(1))=nL_l(z_{n,l})(1+o(1))
\end{align}
and
\begin{align*}
E\tilde{V}_{n,l}^2=nL_l(z_{n,l})(1+o(1)).
\end{align*}
Then for $0<\delta<1$,
\begin{eqnarray}
&&P \left ( \frac{\sum_{l=1}^{m}\lambda_l  \left \{\sum_{i=1}^n g_l(X_i) \right \}^2   }{\max_{1 \leq l \leq m} \lambda_l V^2_{n,l}} \geq x_n^2 \right)\notag\\
&\geq& P \left ( \frac{\sum_{l=1}^{m}\lambda_l  \left \{\sum_{i=1}^n g_l(X_i) \right \}^2   }{\max_{1 \leq l \leq m} \lambda_l V^2_{n,l}} \geq x_n^2, \max_{1\le i\le n}|g_l(X_i)|\le z_{n,l}, ,1\le l\le m\right)\notag\\
&=& P \left ( \frac{\sum_{l=1}^{m}\lambda_l  \left \{\sum_{i=1}^n \tilde{g}_l(X_i) \right \}^2   }{\max_{1 \leq l \leq m} \lambda_l \tilde{V}^2_{n,l}} \geq x_n^2\right)P\left(\max_{1\le i\le n}|g_l(X_i)|\le z_{n,l}, ,1\le l\le m\right)\notag\\
&\ge& P \left ( \frac{\sum_{l=1}^{m}\lambda_l  \left \{\sum_{i=1}^n \tilde{g}_l(X_i) \right \}^2   }{\max_{1 \leq l \leq m} \lambda_l \tilde{V}^2_{n,l}} \geq x_n^2, \tilde{V}^2_{n,l}\le (1+2\delta)\sigma^2_l, 1\le l\le m\right)\notag\\
&&\times P\left(\max_{1\le i\le n}|g_l(X_i)|\le z_{n,l}, ,1\le l\le m\right)\notag\\
&\ge&P \left ( \frac{\sum_{l=1}^{m}\lambda_l  \left \{\sum_{i=1}^n \tilde{g}_l(X_i) \right \}^2   }{\max_{1 \leq l \leq m} \lambda_l \sigma^2_l} \geq (1+2\delta) x_n^2\right)\notag\\
&&\times P\left(\max_{1\le i\le n}|g_l(X_i)|\le z_{n,l}, ,1\le l\le m\right)
-\sum_{l=1}^m P(\tilde{V}^2_{n,l}\ge (1+2\delta)\sigma^2_l).\label{summ}
\end{eqnarray}
Without loss of generality, assume that $\max_{1 \leq l \leq m} \lambda_l \sigma^2_l=\lambda_1 \sigma^2_1$. Then
\begin{align*}
&P \left ( \frac{\sum_{l=1}^{m}\lambda_l  \left \{\sum_{i=1}^n \tilde{g}_l(X_i) \right \}^2   }{\max_{1 \leq l \leq m} \lambda_l \sigma^2_l} \geq (1+2\delta) x_n^2\right)\\
&\ge P \left ( \frac{\lambda_1  \left \{\sum_{i=1}^n \tilde{g}_1(X_i) \right \}^2   }{ \lambda_1 \sigma^2_1} \geq (1+2\delta) x_n^2\right)\\
&\ge P \left (\sum_{i=1}^n \tilde{g}_1(X_i) \geq (1+2\delta)^{1/2}\sigma_1 x_n\right).
\end{align*}
Recall (\ref{nznl}) and (\ref{sigmal}). Take $c=1/x_n$, we have $|\tilde{g}_1(X_1)|\le z_{n,l}=c\sigma_l$.  Therefore, by Theorem 5.2.2 in Stout \cite{Stout}, for any $\gamma>0$,  we have
\begin{align}\label{yntail}
P \left (\sum_{i=1}^n \tilde{g}_1(X_i) \geq (1+2\delta)^{1/2}\sigma_1 x_n\right)\ge \exp(-(x_n^2/2)(1+2\delta)(1+\gamma)).
\end{align}
On the other hand,
\begin{align}
&P\left(\max_{1\le i\le n}|g_l(X_i)|\le z_{n,l}, 1\le l\le m\right)=[P\left(|g_l(X_1)|\le z_{n,l}, 1\le l\le m\right)]^n\notag\\
&=[1-P\left(|g_l(X_1)|\ge z_{n,l}, \exists 1\le l\le m\right)]^n\ge [1-\sum_{l=1}^mP\left(|g_l(X_1)|\ge z_{n,l}\right)]^n\notag\\
&\ge \exp(-2n \sum_{l=1}^mP\left(|g_l(X_1)|\ge z_{n,l}\right))=\exp(-o(x_n^2)).\label{maxg}
\end{align}

We apply the following exponential inequality (see Lemma 2.1, Cs\"{o}rg\H{o}, Lin and Shao \cite{CLS}, see also Griffin and Kuelbs \cite{GK} and Pruitt \cite{P}) for the rest of the proof.
\begin{lemma}\label{expineq}
Let $\xi, \xi_1, \cdots, \xi_n$ be i.i.d. random variables. Then for any $b, v, s>0$,
\begin{align*}
&P\left(\left|\sum_{i=1}^n (\xi_i I(|\xi_i|\le b)-E\xi_i I(|\xi_i|\le b))\right|\ge \frac{ve^v nE\xi_i^2 I(|\xi_i|\le b)}{2b}+\frac{sb}{v}\right)\le 2e^{-s}.
\end{align*}
By (\ref{Egpx}),
\begin{align}\label{power4}
E\tilde{g}^4_{1}(X_i)= o(z_{n,1}^2 L_1(z_{n,1})).
\end{align}
In Lemma \ref{expineq}, we take $\xi_i=\tilde{g}^2_{1}(X_i), s=x_n^2, b=z_{n,1}^2$ and $v=1/\delta$. Notice that $sb/v=\delta\sigma^2_1$  and $\frac{ve^v nE\xi_i^2 I(|\xi_i|\le b)}{2b}=o(\sigma^2_1)$ by (\ref{sigmal}) and (\ref{power4}). Then
\begin{align}
&P(\tilde{V}^2_{n,1}\ge (1+2\delta) \sigma^2_1)=P(\tilde{V}^2_{n,1}-E\tilde{V}^2_{n,1}\ge (1+2\delta)\sigma^2_1-E\tilde{V}^2_{n,1})\notag\\
&\le P\left(\sum_{i=1}^n (\tilde{g}^2_{1}(X_i)-E\tilde{g}^2_{1}(X_i))\ge \delta(1+\delta)\sigma^2_1\right)\notag\\
&\le 2\exp(-x_n^2). \label{Vntail}
\end{align}
Combining (\ref{3.1}), (\ref{summ}), (\ref{yntail}), (\ref{maxg}) and (\ref{Vntail}) and letting $\lambda, \delta\rightarrow 0$, we have
\begin{align*}
P (W_n \geq (1-\varepsilon) x_n^2) \ge \exp(-x_n^2/2).
\end{align*}
\end{lemma}
% --------------------------------------------------------------------------------------------------- section

\section{The upper bound of Theorem \ref{LIL}}

%  ----------------------------------------------------------------------------------- lemma

\begin{lemma}[Lemma 2.3 of Gin\a'{e}, Kwapie\a'{n}, Lata\l a, and Zinn \cite{GKLZ}] \label {GKLZ}
There exists a universal constant $C_3<\infty$ such that for any kernel $h$ and any two sequences of i.i.d. random variables, we have
\begin{eqnarray}
P \bigg ( \max_{k \leq m, l \leq n} \bigg | \sum_{i \leq k, j \leq l} h(X_i, Y_j)  \bigg |  \geq t\bigg )  \leq  C_3 P \bigg (  \bigg | \sum_{i \leq m, j \leq n} h(X_i, Y_j)  \bigg |  \geq t/C_3\bigg ) \nonumber
\end{eqnarray}
for all $m, n \in \mathbb{N}$ and all $t>0$.
\end{lemma}

% ------------------------------------------------------------------------------------- Proof of Theorem 1.2

\begin{proposition}\label {Proposition4.1}
Under the assumptions of Theorem \ref{moderate},
\begin{eqnarray*} \label {P4.1a}
\limsup_{n \rightarrow \infty} \frac{\sum_{l=1}^{\infty}\lambda_l \left (\sum_{i=1}^n g_l(X_i)   \right)^2}{\max_{1 \leq l < \infty} \lambda_l V^2_{n,l} \log \log n} \leq 2   \ \ a.s.
\end{eqnarray*}
Consequently,
\begin{eqnarray*}
\limsup_{n \rightarrow \infty} \frac{W_n}{  \log \log n} \leq 2   \ \ a.s.
\end{eqnarray*}
\end{proposition}
\begin{proof}
Let $x_n\rightarrow \infty$ as $n\rightarrow \infty$. Let $\theta>1$ with $\theta-1$ sufficiently small. For any positive integer $k \in (n, \theta n]$, by similar idea as in (\ref{XY}) with $0<\eta<1$,
\begin{eqnarray}\label {AC4.6}
&& P \Bigg (\max_{n < k \leq  \theta n} \frac{\sum_{l=1}^{\infty}\lambda_l \left (\sum_{i=1}^k g_l(X_i)   \right)^2}{\max_{1 \leq l < \infty} \lambda_l V^2_{k,l} } \geq 2 (1+\eta)^3 x_n^2 \Bigg) \nonumber \\
&\leq& P\Bigg (\frac{\sum_{l=1}^{\infty} \lambda_l   \big (\sum_{i=1}^{n} g_l (X_i)  \big )^2 }{\max_{1 \leq l < \infty}\lambda_l V^2_{n,l}} \geq 2(1-\eta) (1+\eta)^3x_{n}^2 \Bigg) \nonumber \\
&& + P\Bigg (\max_{n <  k \leq  \theta n}\frac{\sum_{l=1}^{\infty} \lambda_l   \big (\sum_{i=n+1}^k g_l (X_i)   \big )^2  }{\max_{1 \leq l < \infty} \lambda_lV^2_{k,l}} \geq \frac{\eta^2 (1+\eta)^3x_{n}^2}{2} \Bigg) \nonumber \\
&:=& H_1 + H_2.
\end{eqnarray}
Notice that \eqref{a1'} implies \eqref{a1}. By \eqref{a1b} and the upper bound of Theorem \ref{moderate},
\begin{eqnarray}\label {AC4.7}
H_1 \leq \exp \left (-(1-\eta)^{3/2} (1+\eta)^3 x^2_{n} \right ).
\end{eqnarray}
Let $0<\delta<1$ be a sufficiently small constant. By (\ref{XY}),
\begin{eqnarray}\label {H123}
&&H_2 \leq P \left (2m\max_{1 \leq l < \infty} \lambda_l V^2_{n,l} \leq (1-\eta) n\sum_{1 \leq l < \infty}\lambda_l L_l(z_{n,l})\right )  \nonumber \\
&& + P\Bigg ( \max_{n< k \leq \theta n}\frac{\sum_{l=1}^{\infty} \lambda_l \left ( \sum_{i=n+1}^k g_l(X_i) I \left (|g_l(X_i)| > \delta \eta z_{n,l} \right )  \right )^2}{ \max_{1 \leq l < \infty}\lambda_l V^2_{k,l}}  \geq \frac{\eta^4 (1+\eta)^3 x_{n}^2}{2} \Bigg ) \nonumber \\
&& + P \Bigg (  \frac{ \max_{n<  k \leq \theta n}\sum_{l=1}^{\infty} \lambda_l \left (\sum_{i=n+1}^k  g_{l}(X_i) I(|g_l(X_i) | \leq \delta \eta z_{n,l}) \right )^2 }{ n\sum_{1 \leq l < \infty}  \lambda_lL_l(z_{n,l}) }   \nonumber\\
&& \geq \frac{(1-\eta)(1-2 \eta) \eta^2  (1+\eta)^3x_{n}^2}{4m}\Bigg ) \nonumber \\
&&:= H_{2,1} +H_{2,2} + H_{2,3}.
\end{eqnarray}
By \eqref{HI21} in Proposition \ref{proposition1},
\begin{eqnarray}\label{H21}
H_{2,1} \leq \exp \big (-2 x^2_{n}\big).
\end{eqnarray}
By Cauchy-Schwarz inequality, for each $k$,
\begin{eqnarray} \label {4.7}
&& \frac{\sum_{l=1}^{\infty} \lambda_l \big \{  \sum_{i=n+1}^{k} |g_l(X_i)| I \left (|g_l(X_i)| > \delta \eta z_{n,l} \right )  \big \}^2}{\max_{1 \leq l <\infty}\lambda_l  V^2_{k,l}} \nonumber\\
&\leq&\frac{\sum_{l=1}^{\infty} \lambda_l  \sum_{i=n+1}^{k} g^2_l(X_i)\sum_{i=n+1}^{k}   I \left (|g_l(X_i)| > \delta \eta z_{n,l} \right )   }{\max_{1 \leq l <\infty}\lambda_l  V^2_{k,l}}.
\end{eqnarray}
By (\ref{a1b}), for some $m$ sufficiently large, the sum of the diagonal terms
\begin{eqnarray}\label {4.8}
\frac{\sum_{l=1}^{\infty} \lambda_l  \sum_{i=n+1}^{k} g^2_l(X_i)  I \left (|g_l(X_i)| > \delta \eta z_{n,l} \right )   }{\max_{1 \leq l <\infty}\lambda_l  V^2_{k,l}} \leq   m.
\end{eqnarray}
By (\ref{4.7}) and (\ref{4.8}),
\begin{eqnarray} \label {4.9}
&&H_{2,2} \leq P\Bigg  (\max_{n< k \leq \theta n} \frac{\sum_{n+1 \leq i \neq j \leq k}\sum_{l=1}^{\infty} \lambda_l   g^2_l(X_i)  I   (|g_l(X_j)| > \delta \eta z_{n,l}) }{\max_{1 \leq l <\infty}\lambda_l  \sum_{i=n+1}^{k} g^2_l(X_i)} \geq \frac{\eta^5 (1+\eta)^3 x_{n}^2}{2} \Bigg )\nonumber\\
&\leq&P\Bigg  ( \max_{n< k \leq \theta n}\frac{\sum_{n+1 \leq i < j \leq k}\sum_{l=1}^{\infty} \lambda_l   g^2_l(X_i)  I   (|g_l(X_j)| >  \delta \eta z_{n,l}) }{\max_{1 \leq l <\infty}\lambda_l  \sum_{i=n+1}^{k} g^2_l(X_i)} \geq \frac{\eta^5 (1+\eta)^3 x_{n}^2}{4} \Bigg )\nonumber\\
&+&P\Bigg  ( \max_{n< k \leq \theta n}\frac{\sum_{n+1 \leq j< i \leq k}\sum_{l=1}^{\infty} \lambda_l   g^2_l(X_i)  I   (|g_l(X_j)| >  \delta \eta z_{n,l}) }{\max_{1 \leq l <\infty}\lambda_l  \sum_{i=n+1}^{k} g^2_l(X_i)} \geq \frac{\eta^5 (1+\eta)^3 x_{n}^2}{4} \Bigg )\nonumber\\
&\leq&P\Bigg  ( \max_{n< k \leq \theta n}\sum_{n+2\le j\le k}\frac{\sum_{n+1 \leq i < j }\sum_{l=1}^{\infty} \lambda_l   g^2_l(X_i)  I   (|g_l(X_j)| > \delta \eta  z_{n,l}) }{\max_{1 \leq l <\infty}\lambda_l  \sum_{n+1 \leq i < j } g^2_l(X_i)} \geq \frac{\eta^5 (1+\eta)^3 x_{n}^2}{4} \Bigg )\nonumber\\
&+&P\Bigg  ( \max_{n< k \leq \theta n}\sum_{n+1 \leq j< k}\frac{\sum_{j< i \leq k}\sum_{l=1}^{\infty} \lambda_l   g^2_l(X_i)  I   (|g_l(X_j)| > \delta \eta  z_{n,l}) }{\max_{1 \leq l <\infty}\lambda_l  \sum_{j< i \leq k} g^2_l(X_i)} \geq \frac{\eta^5 (1+\eta)^3 x_{n}^2}{4} \Bigg )\nonumber\\
&=&H_{2,2,1}+H_{2,2,2}.
\end{eqnarray}
Let
\begin{eqnarray*}
\phi_j = \frac{\sum_{n+1 \leq i < j}\sum_{l=1}^{\infty} \lambda_l   g^2_l(X_i)  I   (|g_l(X_j)| > \delta \eta  z_{n,l}) }{ \max_{1 \leq l <\infty}\lambda_l \sum_{n+1 \leq i < j } g^2_l(X_i)} .
\end{eqnarray*}
Then for any constant $t>0$,
\begin{eqnarray}\label {4.11}
H_{2,2,1} &\leq& \Bigg  ( \sum_{n+2\le j\le [\theta n]}\frac{\sum_{n+1 \leq i < j }\sum_{l=1}^{\infty} \lambda_l   g^2_l(X_i)  I   (|g_l(X_j)| > \delta \eta  z_{n,l}) }{\max_{1 \leq l <\infty}\lambda_l  \sum_{n+1 \leq i < j } g^2_l(X_i)} \geq \frac{\eta^5 (1+\eta)^3 x_{n}^2}{4} \Bigg )  \nonumber\\
&\leq& E e^{t \sum_{j=n+2}^{[\theta n]} \phi_j}e^{-t \eta^5 (1+\eta)^3 x_n^2/4}.
\end{eqnarray}
Let $E_j$ be the expectation of $X_j$ for $n+2 \leq j \leq [\theta n]$. Then
\begin{eqnarray}  \label {4.12}
E e^{t \sum_{j=n+2}^{[\theta n]} \phi_j}=E (e^{t \sum_{j=n+2}^{[\theta n]-1} \phi_j}E_{[\theta n]} e^{ t \phi_{[\theta n]}}).
\end{eqnarray}
Since $|e^s-1| \leq e^{0 \vee s}|s|$ for any $s \in R$ and $0 \leq \phi_{[\theta n]} \leq m$ for some $m$ sufficiently large, then
\begin{eqnarray*}
  \left |E_{[\theta n]} e^{t \phi_{[\theta n]}} - 1 \right | &\leq& e^{m t} t E_{[\theta n]} \phi_{[\theta n]} \nonumber\\
&=& \frac{e^{m t} t \sum_{n+1 \leq   i < [\theta n]}\sum_{l=1}^{\infty} \lambda_l   g^2_l(X_i)   P   (|g_l(X_{[\theta n]})| >\delta \eta  z_{n,l}) }{ \max_{1 \leq l <\infty}\lambda_l \sum_{ n+1 \leq   i <[\theta n]} g^2_l(X_i)}.\nonumber\\
\end{eqnarray*}
By (\ref{Pgxa}) and (\ref{nznl}), we have $P   (|g_l(X_{[\theta n]})| >\delta \eta  z_{n,l}) =o(x_n^2/n)$. Then together with (\ref{a1b}),
\begin{eqnarray} \label {4.14}
E_{[\theta n]} e^{t   \phi_{[\theta n]}} = 1 + o(x_n^2/n) = e^{o(x_n^2/n)}.
\end{eqnarray}
Applying (\ref{4.14}) to (\ref{4.12}), we have
\begin{eqnarray*}
E e^{t \sum_{j=2}^{[\theta n]} \phi_j}= e^{o(x_n^2/n)} E e^{t \sum_{j=n+2}^{ [\theta n]-1} \phi_j} .
\end{eqnarray*}
Similarly,
\begin{eqnarray*}
E e^{t \sum_{j=n+2}^{[\theta n]-1} \phi_j}&=& E (e^{t \sum_{j=n+2}^{[\theta n]-2} \phi_j}E_{n-1} e^{t\phi_{[\theta n]-1}})\nonumber\\
&=& e^{o(x_n^2/n)} E e^{t\sum_{j=n+2}^{[\theta n]-2} \phi_j} .
\end{eqnarray*}
Continue this process from $X_{[\theta n]}$ to $X_{n+1}$, we conclude
\begin{eqnarray} \label {4.17}
E e^{t \sum_{j=n+2}^{ [\theta n]} \phi_j}= e^{[(\theta-1) n] \times o(x_n^2/n)} =e^{o(x^2_n)} .
\end{eqnarray}
Applying (\ref{4.17}) to (\ref{4.11}) and letting $t=8/(\eta^5 (1+\eta)^3)$, we have
\begin{eqnarray}\label{4.18}
H_{2,2,1} \leq \exp (-2x_n^2).
\end{eqnarray}
To estimate $H_{2,2,2}$, let
\begin{eqnarray*}\label {4.20}
\psi_{j, k} = \frac{\sum_{j < i \le k}\sum_{l=1}^{\infty} \lambda_l   g^2_l(X_i)  I   (|g_l(X_j)| > \delta \eta  z_{n,l}) }{ \max_{1 \leq l <\infty}\lambda_l \sum_{j < i \le k} g^2_l(X_i)} .
\end{eqnarray*}
Then for any constant $t>0$,
\begin{eqnarray}\label {4.21}
H_{2,2,2} &\le& P \Bigg  ( \sum_{n+1\le j < [\theta n]}\max_{j< k \leq \theta n}\psi_{j, k} \geq \frac{\eta^5 (1+\eta)^3 x_{n}^2}{4} \Bigg )  \nonumber\\
&\leq& E e^{t \sum_{j=n+1}^{[\theta n]-1} \max_{j< k \leq \theta n} \psi_{j, k}}e^{-t \eta^5 (1+\eta)^3 x_n^2/4}.
\end{eqnarray}
Let $E_j$ be the expectation of $X_j$ for $n+1 \leq j \leq [\theta n]$. Note $k > j$. Then
\begin{eqnarray} \label {4.22}
&&E e^{t \sum_{j=n+1}^{[\theta n]-1} \max_{j< k \leq \theta n}\psi_{j,k}}\notag\\
&&=E (e^{t \sum_{j=n+2}^{[\theta n]-1}\max_{j< k \leq \theta n}  \psi_{j, k}}E_{n+1} e^{ t   \max_{n+1< k \leq \theta n} \psi_{n+1, k}}).
\end{eqnarray}
Observe that
\begin{eqnarray}\label {4.201}
\psi_{j, k} = \frac{\sum_{l=1}^{\infty} \lambda_l  \left(\sum_{j < i \le k}g^2_l(X_i)\right)  I   (|g_l(X_j)| > \delta \eta  z_{n,l}) }{ \max_{1 \leq l <\infty}\lambda_l \sum_{j < i \le k} g^2_l(X_i)} .
\end{eqnarray}
Then by (\ref{a1b}), $0 \leq \psi_{n+1,k} \leq m$ for some $m$ sufficiently large.
Since $|e^s-1| \leq e^{0 \vee s}|s|$ for any $s \in R$,
\begin{eqnarray} \label {4.23}
 && \left |E_{n+1} e^{t  \max_{n+1< k \leq \theta n}\psi_{n+1,k}} - 1 \right | \leq e^{m t} t E_{n+1}  \max_{n+1< k \leq \theta n}  \psi_{n+1, k}.
\end{eqnarray}

Under assumption \eqref{a1'},  for each  $l \in [1, \infty)$,
\begin{eqnarray*}
\lambda_l V^2_{n,l} = \sum_{i=1}^n \lambda_l g_l^2(X_i) \leq c_l \sum_{\nu=1}^{\infty} \sum_{i=1}^n \lambda_\nu g_\nu^2(X_i)=  c_l \sum_{\nu=1}^{\infty} \lambda_\nu V^2_{n,\nu}.
\end{eqnarray*}
Recall that (\ref{a1b}),
then for each $l \in [1, \infty)$,
\begin{eqnarray*}
\frac{\lambda_l V^2_{n,l}}{\max_{1 \leq l <\infty}\lambda_l V^2_{n,l}}\leq \frac{m c_l }{1-\varepsilon}.
\end{eqnarray*}
Hence by \eqref{4.201},
\begin{eqnarray*}
\psi_{j,k}
\leq \frac{m}{1-\varepsilon} \sum_{l=1}^{\infty}c_l  I   (|g_l(X_j)| > \delta \eta  z_{n,l}).
\end{eqnarray*}
Then
\begin{eqnarray*}
E_{n+1} \max_{n+1< k \leq \theta n} \psi_{n+1, k} \leq  \frac{m}{1-\varepsilon} \sum_{l=1}^{\infty}c_l  P   (|g_l(X_{n+1})| > \delta \eta  z_{n,l}).
\end{eqnarray*}
By (\ref{Pgxa}) and (\ref{nznl}), we have $P   (|g_l(X_{n+1})| >\delta \eta  z_{n,l}) =o(x_n^2/n)$. Then together with  \eqref{a1'},
\begin{eqnarray}\label {4.25}
E_{n+1} \max_{n+1< k \leq \theta n}\psi_{n+1, k}=o(x_n^2/n).
\end{eqnarray}
Then by (\ref{4.23}) and (\ref{4.25}),
\begin{eqnarray*}
E_{n+1} e^{t  \max_{n+1< k \leq \theta n}  \psi_{n+1,k}} = 1 + o(x_n^2/n) = e^{o(x_n^2/n)}.
\end{eqnarray*}
Continue this process from $j=n+2$ to $j=[\theta n]-1$ and by (\ref{4.22}),
\begin{eqnarray}\label {4.26}
E e^{t \sum_{j=n+1}^{[\theta n]-1} \max_{j< k \leq \theta n} \psi_{j, k}}=e^{o(x_n^2)}.
\end{eqnarray}
Applying (\ref{4.26}) to (\ref{4.21}) and letting $t=8/(\eta^5 (1+\eta)^3)$, we have
\begin{eqnarray}\label {4.28}
H_{2,2,2} \leq \exp (-2x_n^2).
\end{eqnarray}
By (\ref{4.9}), (\ref{4.18}) and (\ref{4.28}),
\begin{eqnarray}\label {H22}
H_{2,2} \leq 2 \exp (-2x_n^2).
\end{eqnarray}
By the definition of $H_{2,3}$ in (\ref{H123}), and by Lemma \ref{GKLZ}, there is a constant $0<C' < \infty$ such that
\begin{eqnarray}
H_{2,3}\leq C' P \Bigg (  \frac{ \sum_{l=1}^{\infty} \lambda_l \big (\sum_{i=n+1}^{[\theta n]}  g_{l}(X_i)I(|g_l(X_i) | \leq \delta \eta z_{n,l})   \big )^2 }{ n  \sum_{l=1}^{\infty} \lambda_l L_l(z_{n,l}) } \geq \frac{(1-3\eta)\eta^2 x_{n}^2}{4mC'}\Bigg ). \nonumber
\end{eqnarray}
Similar to (\ref{I2.36}),
\begin{eqnarray}
H_{2,3}&\leq& C' P \Bigg  (  \frac{ \sum_{l=1}^{\infty} \lambda_l \left (\sum_{i=n+1}^{[\theta n]} \left \{g_{l}(X_i)I(|g_l(X_i) | \leq \delta \eta z_{n,l}) -E  g_{l}(X_i)I(|g_l(X_i) | \leq \delta \eta z_{n,l})   \right \} \right )^2 }{ n  \sum_{l=1}^{\infty}\lambda_l L_l(z_{n,l}) }\nonumber \\
&& \geq \frac{(1-3\eta)^2\eta^2 x_{n}^2}{4m C'}\Bigg ). \nonumber
\end{eqnarray}
By the decoupling version of (\ref{AC2.33}) in Proposition \ref{supfV},
\begin{eqnarray}\label {H23}
H_{2,3} \leq C' \exp \left (-2x^2_{n} \right).
\end{eqnarray}
Combining (\ref{H123}), \eqref{H21}, \eqref{H22} and \eqref{H23}, we have
\begin{eqnarray}\label {H4.10p}
H_2 \leq (3+ C') \exp \left (-2x^2_{n} \right).
\end{eqnarray}
By (\ref{AC4.6}), (\ref{AC4.7}) and (\ref{H4.10p}),
\begin{eqnarray*}
  &&P \Bigg (\max_{n < k \leq  \theta n} \frac{\sum_{l=1}^{\infty}\lambda_l \left (\sum_{i=1}^k g_l(X_i)   \right)^2}{\max_{1 \leq l < \infty} \lambda_l V^2_{k,l}  } \geq 2 (1+\eta)^3 x_n^2 \Bigg)\\
  &\leq& \exp \left (-(1-\eta)^{7/4} (1+\eta)^3 x^2_{n} \right ).
\end{eqnarray*}
Let $n=[\theta^j]$ for some $j\in \mathbb{N}$. We have
\begin{eqnarray}
  && P \Bigg (\max_{\theta^j < k \leq  \theta^{j+1}} \frac{\sum_{l=1}^{\infty}\lambda_l \left (\sum_{i=1}^k g_l(X_i)   \right)^2}{\max_{1 \leq l < \infty} \lambda_l V^2_{k,l} }\geq 2 (1+\eta)^3 x_{[\theta^j]}^2 \Bigg) \nonumber\\
  && \leq \exp \left (-(1-\eta)^{7/4} (1+\eta)^3 x^2_{[\theta^j]} \right ).\nonumber
\end{eqnarray}
 Let $x_n^2 =\log \log n$. Then
\begin{eqnarray*}
&& \sum_{j=1}^{\infty}  P \Bigg (\max_{\theta^j < k \leq  \theta^{j+1}} \frac{\sum_{l=1}^{\infty}\lambda_l \left (\sum_{i=1}^k g_l(X_i)   \right)^2}{\max_{1 \leq l < \infty} \lambda_l V^2_{k,l} \log \log k  } \geq 2 (1+\eta)^3   \Bigg) \nonumber\\
&\leq& \sum_{j=1}^{\infty}  P \Bigg (\max_{\theta^j < k \leq  \theta^{j+1}} \frac{\sum_{l=1}^{\infty}\lambda_l \left (\sum_{i=1}^k g_l(X_i)   \right)^2}{ \max_{1 \leq l < \infty}  \lambda_lV^2_{k,l} } \geq 2 (1+\eta)^3x_{[\theta^j]}^2  \Bigg) \nonumber\\
&\leq& \sum_{j=1}^{\infty}\exp \left (-(1-\eta)^{7/4} (1+\eta)^3 \log \log [\theta^j] \right ) \nonumber\\
&\leq& K \sum_{j=1}^{\infty}\exp \left (-(1-\eta)^2 (1+\eta)^3   \log j \right ) < \infty.
\end{eqnarray*}
By Borel-Cantelli lemma,
\begin{eqnarray*}
&& \limsup_{n \rightarrow \infty} \frac{\sum_{l=1}^{\infty}\lambda_l \big (\sum_{i=1}^n g_l(X_i) \big )^2}{\max_{1 \leq l < \infty}\lambda_l V^2_{n,l}\log \log n }  \leq \ 2 \ \ a.s.
\end{eqnarray*}
\end{proof}
% ---------------------------------------------------------------------------------------- section

%%%%%%%%%%%%%%%%%%%%%%%%%%%%%%%%%%%%
\section{The lower bound of Theorem \ref{LIL}}

%  ----------------------------------------------------------------------------------- proposition

\begin{proof}
By the definition of $W_n$,
\begin{eqnarray*}
\frac{W_n}{\log \log n} &=& \frac{\sum_{l=1}^{\infty} \lambda_l \bigg (  \Big (\sum_{i=1}^n g_l (X_i)  \Big  )^2 - \sum_{i=1}^n g_l^2(X_i) \bigg )}{\max_{1 \leq l < \infty}\lambda_l V^2_{n,l} \log \log n}\nonumber\\
&=& \frac{\sum_{l=1}^{\infty} \lambda_l    \Big (\sum_{i=1}^n g_l (X_i)  \Big  )^2 }{ \max_{1 \leq l < \infty}\lambda_l V^2_{n,l} \log \log n}- \frac{\sum_{l=1}^{\infty} \lambda_l    V_{n,l}^2}{ \max_{1 \leq l < \infty}\lambda_l V^2_{n,l} \log \log n}.
\end{eqnarray*}
By (\ref{a1b}), \ $\sum_{l=1}^{\infty} \lambda_l    V_{n,l}^2/( \max_{1 \leq l < \infty}\lambda_l V^2_{n,l} \log \log n) \leq m/((1-\varepsilon)\log \log n) \rightarrow 0$ as $n \rightarrow \infty$, then
\begin{eqnarray*}
\frac{W_n}{\log \log n} \geq \frac{\sum_{l=1}^{\infty} \lambda_l    \Big (\sum_{i=1}^n g_l (X_i)  \Big  )^2 }{ \max_{1 \leq l < \infty}\lambda_l V^2_{n,l} \log \log n}.
\end{eqnarray*}
%If $\max_{1 \leq l < \infty}\lambda_l V^2_{n,l}=\lambda_k V^2_{n,k}$ for some $k \in [1, \infty)$, then
Then
\begin{eqnarray*}
\frac{W_n}{\log \log n} \geq \sum_{k=1}^\infty\frac{   \Big (\sum_{i=1}^n g_k (X_i)  \Big  )^2 }{  V^2_{n,k} \log \log n}I_{k=min\{j: \max_{1 \leq l < \infty}\lambda_l V^2_{n,l}=\lambda_j V^2_{n,j}\}}.
\end{eqnarray*}
Hence by (\ref{1.1aa}),
\begin{eqnarray*}
\limsup_{n \rightarrow \infty}\frac{W_n}{\log \log n}  \geq 2 \ \ a.s.
\end{eqnarray*}
\end{proof}

%**************************************************************************************%
%**********************bibliography**************************************************%

\end{document}